\documentclass[11pt]{amsart}

\title[Time-frequency concentration problems for the Cohen class via QHA]{A quantum harmonic analysis approach to nonlinear time-frequency concentration}
\author{Erling A. T. Svela}\thanks{Norwegian University of Science and Technology, \url{erling.a.t.svela@ntnu.no}}
\author{S. Ivan Trapasso}\thanks{Politecnico di Torino, \url{salvatoreivan.trapasso@polito.it}}
\date{}

\usepackage{graphicx} 

\usepackage{amsthm}
\usepackage{amssymb}
\usepackage{amsmath,mathtools}
\mathtoolsset{showonlyrefs}
\usepackage{xcolor}
\usepackage[style=ext-numeric,
            sorting=nyt,
            sortcites=true,
            maxnames=50,
            giveninits=true,
            articlein=false]{biblatex}
\usepackage{physics}
\bibliography{Bibliography}
\usepackage{mathrsfs}
\usepackage{hyperref}
\usepackage[title]{appendix}
\usepackage{microtype}

\usepackage{nicematrix}

\calclayout

\usepackage{geometry}
\newtheorem{theorem}{Theorem}[section]
\newtheorem{prop}[theorem]{Proposition}
\newtheorem{corollary}[theorem]{Corollary}
\newtheorem{lemma}[theorem]{Lemma} 

\theoremstyle{definition}
\newtheorem{definition}[theorem]{Definition} 
\newtheorem{example}{Example}[section]

\theoremstyle{remark}
\newtheorem{remark}[theorem]{Remark} 

\usepackage{verbatim}

\newcommand{\C}{\mathbb{C}}
\newcommand{\R}{\mathbb{R}}

\newcommand{\Z}{\mathbb{Z}}
\newcommand{\N}{\mathbb{N}}

\renewcommand{\S}{\mathcal{S}}

\newcommand{\D}{\mathcal{D}}

\DeclareFontFamily{U}{mathx}{}
\DeclareFontShape{U}{mathx}{m}{n}{<-> mathx10}{}
\DeclareSymbolFont{mathx}{U}{mathx}{m}{n}
\DeclareMathAccent{\widehat}{0}{mathx}{"70}
\DeclareMathAccent{\widecheck}{0}{mathx}{"71}

\newcommand{\cF}{\mathcal{F}}
\newcommand{\cB}{\mathcal{B}}
\newcommand{\Cu}{\mathcal{C}_{\mathrm{u}}^\alpha}

\def\cS{\mathcal{S}}
\def\Schw{\mathscr{S}}
\def\sS{\mathscr{S}}
\def\cK{\mathcal{K}}
\def\cB{\mathcal{B}}
\def\cH{\mathcal{H}}

\def\rd{\mathbb{R}^d}
\def\rdd{\mathbb{R}^{2d}}

\def\cBsup{\cB_{\mathrm{sup}}(p,\Omega)}
\def\cBwc{\cB_{\mathrm{wc}}(p,\Omega)}
\def\cBwtn{\cB_{\mathrm{wtn}}(p,\Omega)}
\def\cBubd{\cB_{\mathrm{ubd}}(p,\Omega)}
\def\cBpt{\cB_{\mathrm{pt}}}
\def\cG{\mathcal G(p,\Omega)}
\def\cTatt{\mathcal T_{\mathrm{att}}(p,\Omega)}
\def\cTloss{\mathcal T_{\mathrm{loss}}(p,\Omega)}
\def\NRA{\mathrm{NRA}}

\renewcommand{\dd}[1]{\,\mathrm{d} #1}

\newcommand{\smo}{\setminus\{0\}}

\newcommand{\BJ}{W_{\mathrm{BJ}}}

\begin{document}
\begin{abstract}
We study nonlinear concentration problems for time-frequency distributions in the Cohen class. Using recent techniques from quantum harmonic analysis (QHA) we provide both positive and negative results, such as sufficient conditions for the existence of optimizers in terms of the ``window operator'' and explicit examples where the supremum is never attained. We also study the structural properties of window operators, in particular operators that yield weakly continuous concentration functionals and operators for which the nonlinear concentration problem admits an optimizer, also beyond the Heisenberg representation. We then consider generalizations to the study of concentration problems for phase space representations of operators. We consider generalized Husimi distributions via quantum convolution, and their optimization problem when optimizing over Hilbert--Schmidt and density operators. Lastly, we consider representations of operators on double phase space, in the spirit of quantum time-frequency analysis, and give a full solution in terms of the Weyl symbols.
\end{abstract}
\subjclass[2020]{81S30, 42B10, 49Q10, 49R05, 94A12} 
\keywords{Nonlinear time-frequency concentration, optimization, Cohen's class, quantum harmonic analysis, localization}

\maketitle

\tableofcontents

\section{Introduction}
\subsection{Concentration problems for the Cohen class}
Phase space representations of functions and operators are of great use in several problems of signal analysis and quantum mechanics. The archetypical example, and historically also the first one, is the \textit{Wigner distribution}, which is defined for $f \in L^2(\rd)$ by 
\begin{align*}
    Wf(z)=\int_{\R^d}e^{-2\pi i \xi \cdot y} f\left(x+\frac{y}{2}\right)\overline{f\left(x-\frac{y}{2}\right)} \dd{y}, \qquad z=(x,\xi) \in \rdd.
\end{align*} This can be rightfully viewed as the restriction to the diagonal of the sesquilinear polarization $W(f,g)$, with $f,g \in L^2(\rd)$, usually known as the \textit{cross-Wigner distribution}:
\[
W(f,g)(z)=\int_{\R^d}e^{-2\pi i \xi \cdot y} f\left(x+\frac{y}{2}\right)\overline{g\left(x-\frac{y}{2}\right)} \dd{y} \qquad z=(x,\xi) \in \rdd.
\]
Introduced by Wigner in 1932~\cite{Wigner}, the Wigner distribution was ultimately an attempt at making a joint probability density between \(|f|\) and \(|\hat{f}|\) on phase space. As such, it has indeed many of the expected features, including the marginal properties
\begin{align*}
    \int_{\R^d} Wf(x,\xi) \dd{\xi}=|f(x)|^2,\quad \int_{\R^d} Wf(x,\xi) \dd{x}=|\hat{f}(\xi)|^2,
\end{align*} at least when such relations are meaningful, e.g., if $f,\hat{f}\in L^1(\rd)\cap L^2(\rd)$. However, the Wigner distribution also suffers from defects that prevent it from being interpreted as a genuine phase space energy density. Notably, the Wigner distribution is in general not a positive function. Moreover, the Wigner distribution of a sum \(W(f+g)\) can look significantly different from the sum of the Wigner distributions \(Wf+Wg\), due to the substantial interferences represented by the cross-terms \(W(f,g)\) and \(W(g,f)\). 

This has led to the study of different time-frequency distributions, where one tries to keep the good features of the Wigner distribution while mitigating the negative ones as much as possible. In keeping with this spirit, a natural way to tame negativity and interferences is to carry out some form of local averaging in phase space, for instance in terms of convolution with a suitably designed smoothing kernel. This is precisely the rationale behind the \textit{Cohen class} of quadratic time-frequency distributions \cite{Cohen66,cohen1995time}:
\[ Q_\sigma (f,g) \coloneqq W(f,g)*\sigma, \qquad \sigma \in \mathscr{S}'(\R^{2d}), \, f,g \in L^2(\rd). \] 
We abbreviate $Q_\sigma f$ in the case where $g=f$. Like the Wigner transform, these distributions are \textit{covariant} with respect to time-frequency shifts: Setting $\pi(z)f(y)=e^{2\pi i \xi \cdot y} f(y-x)$ for the phase space shift of along $z=(x,\xi) \in \rdd$, this means that 
\[ Q_\sigma (\pi(z)f)(w) = Q_\sigma f(w-z), \qquad w \in \rdd.  \]
In fact, it can be shown~\cite{Grochenig} that any covariant quadratic time-frequency distribution which satisfies a certain continuity assumption (that is, weak-* continuity in each argument of the corresponding sesquilinear distribution, see~\cite{GSMetaplectic}) falls within Cohen's class. While one may be inclined to believe that careful kernel design is eventually able to isolate an optimal time-frequency distribution in this family, many of the desirable properties turn out to be mutually incompatible~\cite{JanssenSurvey}. This justifies the study of Cohen's class in its entirety. 

As part of a recent series by the authors and collaborators \cite{Nicola-Romero-Trapasso-2022,Stra-Svela-Trapasso-2025,Stra-Svela-Trapasso-2026}, this paper is concerned with nonlinear time-frequency concentration problems for Cohen's class distributions: Given a phase space subset $\Omega\subset\rdd$ of positive finite Lebesgue measure, $1 \le p \le \infty$, and a time-frequency distribution \(Q_\sigma\) in Cohen’s class, we are interested in guaranteeing the existence of optimizers for the problem
\begin{align}\label{smallProblem}
    \sup_{f\in L^2(\R^d)\setminus\{0\}}\frac{\norm{Q_\sigma f}_{L^p(\Omega)}}{\|f\|_{L^2}^2}.
\end{align} Concentration problems of many types are broadly studied in time-frequency analysis, as they can be thought of as quantitative formulations of the uncertainty principle --- in fact, we may interpret an optimizer of the quotient~\eqref{smallProblem} as an uncertainty minimizer, if uncertainty is measured in terms of the \textit{local} \(L^p\)-norm (see \cite{Lieb} for the global Wigner case, and also \cite{lerner} for related local problems). 

As anticipated, this problem has been addressed only recently for the most popular distributions, namely in \cite{Nicola-Romero-Trapasso-2022} for the ambiguity transform, in \cite{Stra-Svela-Trapasso-2025} for (generalized) Wigner distribution and in \cite{Stra-Svela-Trapasso-2026} for the Born--Jordan distribution (see \cite{CdGN17} for the explicit expression of $\sigma$, whose symplectic Fourier transform reads $\operatorname{sinc}(\pi x \cdot\xi)$). Each of these results comes with distinctive difficulties and peculiar behaviors, although the underlying approach relies on a common template exploiting concentration compactness to isolate and control the possible loss of phase space mass in maximizing sequences. 

Motivated by the suggestive evidences emerging in these investigations, we decided to complement our analysis by focusing on the Cohen class by means of a more systematic approach. We stress that this is far from a trivial generalization of what is already known about the Wigner case, as even the tiniest modification (e.g., the $\tau$-Wigner distributions) may break the delicate structure of the proofs and produce new interference patterns, which in turn are expected to reflect in different concentration behaviors for \eqref{smallProblem} --- see \cite{Stra-Svela-Trapasso-2025}. In this connection, let us highlight that the main technical difficulties in the Wigner case stemmed precisely from the energy contribution of constructive interference phenomena that ultimately occur due to time-frequency covariance of this phase-space representation --- compare with the case of the ambiguity transform in \cite{Nicola-Romero-Trapasso-2022}, which is invariant instead. Since every Cohen's class distribution inherits covariance from the Wigner one, this might naively suggest that such (already non-trivial) complications may only get worse. In fact, we show below that this is not always the case, in line with the heuristic mitigating effect of kernel convolution. In general, the Cohen class is flexible enough to give us the opportunity to detect clusters of sharply different phenomena. 

\subsection{The QHA viewpoint} A relevant difference with respect to the other articles of the series is in the technical apparatus. Indeed, here we make extensive use of the framework of quantum harmonic analysis (QHA), first introduced by Werner in \cite{Werner} and recently studied by Luef and collaborators~\cite{Bible1,FLW26,BBLS22,OpSTFT,LMcN025,FHL24} in the context of time-frequency analysis. In particular, quantum harmonic analysis has turned out to be particularly suited to investigations related to the Cohen class ~\cite{CohenClass,DaubechiesExt,LS20}. The main purpose of QHA is to parallel harmonic analysis of functions and distributions at the level of operators, hence providing operator-theoretic versions of notions like convolution and the Fourier transform. For instance, the convolution between trace class operators $T,S$ on $L^2(\rd)$ is the function defined by
\[ T\star S(z) \coloneqq \tr\big(T\alpha_z(\widecheck{S})\big), \qquad z \in \rdd, \]
where we introduced the \textit{operator shift} $\alpha_z(A)=\pi(z)A\pi(z)^*$ and the \textit{operator reflection} $\widecheck{A}=PAP$, that is the conjugation with the parity operator \(P\) defined by \(Pf(y)=f(-y) \). This definition extends to more general families of operators, including Schatten classes --- we address the reader to Section~\ref{sec:Prelim} below for a more detailed outline. In any case, this program is far from being a mere exercise in style aimed at recovering known results from a different viewpoint: Unlocking the paradigm of Fourier analysis in the non-commutative world of operators yields a considerably powerful machinery. We have already witnessed new solutions to old (function-level) problems originating from these additional degrees of freedom, as well as inherently original problems which have attracted attention in the last few years~\cite{LS25,FHL24,Luef-Skrettingland-2021,FLW26,S20}. 

Let us now illustrate how these remarks are particularly relevant to concentration problems discussed before. The key conceptual leap allowed by QHA in this context comes from a result in \cite{CohenClass}, where the authors showed that action on $f \in \sS(\rd)$ of distributions in Cohen's class like \(Q_\sigma\) actually coincide with the operator convolution between the rank-one projection $f \otimes f \colon L^2(\rd) \ni h \mapsto \langle h,f\rangle f$  and the Weyl pseudodifferential operator with symbol $\sigma$ (cf.\ Section \ref{sec:Prelim}), that is 
\[ Q_\sigma f= Wf * \sigma = (f\otimes f)\star L_\sigma, \qquad f \in \sS(\rd), \quad \sigma \in \sS'(\rdd).  \] 
A genuine QHA angle on this problem requires shifting the main focus from the Cohen kernel $\sigma$ to the corresponding operator counterpart, and rather consider the Cohen-type distributions like
\[ Q_S f = (f\otimes f)\star \widecheck{S} = \tr\bigl((f\otimes f)\alpha_z(S)\bigr)=\langle \alpha_z(S)f,f\rangle  \] indexed by smoothing operators $S$ with Weyl symbol in $\sS'(\rdd)$, so that $Q_\sigma f$ corresponds to $Q_Sf$ with $S=\widecheck{L_\sigma}$. The advantages of this reformulation lie in that rank-one operators are generally better behaved than the Wigner function. For instance, while the standard projection \(f\otimes f \) is trace class for all \(f\in L^2(\R^d)\), the parallel condition \(Wf\in L^1(\R^{2d})\) requires the stronger regularity assumptions on $f$, which must belong in Feichtinger's algebra \(M^1\)~\cite{Fei}. Similarly, since \(f\otimes f\) is positive for all \(f\in L^2(\R^d)\), \(Q_Sf\) will be positive as long as \(S\) is. On the other hand, the positivity of \(Q_\sigma\) is an extremely delicate issue, ultimately equivalent to the positivity of the Weyl quantization of \(\sigma\) and thus linked to the subtle KLM conditions~\cite{Kastler,LM1,LM2}. For additional examples in this vein, the interested reader may consult \cite[Sections 7--9]{CohenClass}.

We will therefore be concerned with the following optimization problem: For $1 \le p \le \infty$, $\Omega \subset \rdd$ with $0 < \abs{\Omega} < \infty$ and $S \in \cB(L^2(\rd))$, 
\begin{align}\label{Problem}
    \Lambda_{p,\Omega}(S) \coloneqq \sup_{f\in L^2(\R^d)\setminus\{0\}}\frac{\norm{Q_S f}_{L^p(\Omega)}}{\|f\|_{L^2}^2} = \sup_{\substack{f\in L^2(\rd) \\ \norm{f}_{L^2}^2=1}} \norm{Q_S f}_{L^p(\Omega)}. 
\end{align}
It is easy to realize that $0 \le \Lambda_{p,\Omega}(S) \le \abs{\Omega}^{1/p} \norm{S}_{\cB}$, with $\Lambda_{p,\Omega}(S)=0 \iff S=0$ (cf.\ Lemma \ref{lem:nonzero-window-nonvanishing} below). While \eqref{Problem} is exactly the same problem as \eqref{smallProblem}, the QHA viewpoint is arguably better suited here, due to the better summability and positivity properties of the rank-one operators \(f\otimes f\) compared to the Wigner distribution \(Wf\). In fact, the special case \(p=1\) and positive $S$ has already been successfully studied using QHA in~\cite{CohenClass,DaubechiesExt}. The general case is considerably harder, as \eqref{Problem} can no longer be recast as an eigenvalue problem due to the nonlinearity. 

As anticipated, the standard approach so far has been to treat each Cohen's class distribution separately. In some cases, such as the rank-one operator \(S=g\otimes g\) associated with the standard Gaussian $g(y)=e^{-\pi \abs{y}^2}$, this has led to some remarkable results~\cite{FaberKrahn} (including shape optimization for $\Omega$), with the obvious disadvantage that the inherently complex-analytic machinery unlocked by the Gaussian window hardly generalizes to other Cohen's class distributions. The pure Wigner case handled in ~\cite{Stra-Svela-Trapasso-2025} corresponds to $\sigma = \delta$, that is $S=2^dP$ in the QHA perspective. Nevertheless, one should take into account that the transition from \(Q_\sigma\) to \(Q_S\) also represents a change in perspective on which Cohen's class distribution is the ``fundamental'' one. Indeed, while the Wigner distribution (\(\sigma=\delta\)) is the building block of Cohen's class, from the QHA viewpoint this role is played by the elementary tensors \(S=h\otimes g\), hence the (generalized) spectrograms. 

\subsection{Main results} Generally speaking, a straightforward way to show the existence of an optimizer for a given functional is the direct method of calculus of variations \cite{CVBook}. In our setting, where the quotient in~\eqref{Problem} is homogeneous of degree \(0\), this scheme amounts to a generalization of the extreme value theorem: If \(X\) denotes the closed unit ball of \(L^2(\R^d)\), which is sequentially weakly compact by the Banach--Alaoglu theorem and reflexivity, a sufficient condition for the concentration functional 
\begin{align}
   J_S \colon X\to[0,\infty], \qquad J_S(f)=J_{S,\Omega}^p(f) \coloneqq \norm{Q_Sf}_{L^p(\Omega)}
\end{align} 
to have a maximizer is being weakly upper semicontinuous. We explicitly highlight that, despite being a natural approach, the direct method cannot be invoked in the cases of ambiguity \cite{Nicola-Romero-Trapasso-2022} and Wigner concentration \cite{Stra-Svela-Trapasso-2025} due to failure of semicontinuity, while continuity of the concentration functional holds in the case of the spectrograms and, less trivially, in the Born--Jordan scenario \cite{Stra-Svela-Trapasso-2025,Stra-Svela-Trapasso-2026}. 

An interesting problem is thus to investigate the interplay between the regularity of \(J_S\) and that of the window operator \(S\). In this regard, we are able to extend the aforementioned findings for the (subcritical) Born--Jordan case to a completely unrelated class of operator windows, namely compact ones. 
\begin{theorem}\label{CompactExistence}
 Let \(S \in \cK(L^2(\rd))\) be a compact operator, \(p \in [1,\infty)\), and \(\Omega\subset\R^{2d}\) be such that $0<\abs{\Omega}<\infty$. The concentration functional 
 \begin{align*}
        J_S \colon L^2(\rd) \to [0,+\infty), \qquad J_S(f) = \left(\int_{\Omega}|Q_Sf(z)|^p \dd{z}\right)^{1/p} 
    \end{align*}
    is sequentially weakly continuous on \(L^2(\R^d)\). As a consequence, the supremum $\Lambda_{p,\Omega}(S)$ in \eqref{Problem} is attained. Furthermore, if $S \ne 0$ then any normalized maximizing sequence has a subsequence strongly converging to a maximizer in \(L^2(\R^d)\).
\end{theorem}

The proof is remarkably short, because QHA exposes the decisive compactness mechanism: If $f_n\rightharpoonup f$, then $f_n\otimes f_n\to f\otimes f$ weak-* in trace class, and compactness of $S$ converts this into pointwise convergence of $Q_Sf_n$; dominated convergence then gives local $L^p$ convergence. The endpoint case $p=\infty$ is treated in Section \ref{sec-linfty}, where we prove that the optimal value $\Lambda_{\infty,\Omega}(S)$ coincides with the numerical radius $w(S)$ of $S \in \cB$. In particular, if $S$ is a numerical-radius-attaining operator \cite{berg}, the supremum $\Lambda_{\infty,\Omega}(S)$ is attained as well. 

Let us emphasize once more that the existence result for compact operators is far from being a general property in the Cohen class --- semicontinuity of the concentration functional fails for the Wigner distribution, as well as its closest relatives, the \(\tau\)-Wigner distributions. To the best of our knowledge, continuity of \(J_S\) was previously only known for the spectrograms (\(S=g\otimes h\))~\cite{Nicola-Romero-Trapasso-2022} and more recently, the Born--Jordan distribution~\cite{Stra-Svela-Trapasso-2025}. Theorem~\ref{CompactExistence} thus represents a significant step in the theory, as it provides an infinite class of new Cohen's class distributions whose concentration functional is weakly continuous. It also illustrates another advantage of Cohen's class: smoothing the Wigner distribution via a convolution can upgrade its properties, even on a topological level. 

The compact theorem is of course not the whole story, as several noncompact windows also admit optimizers. We prove this for positive compact perturbations of the identity and, more generally, for positive windows whose local mixed-state localization operator
\[ H_{\Omega,S}=\int_\Omega \alpha_z(S)\dd{z} \]
is compact. This identifies a genuinely \textit{local} compactness mechanism --- informally, the window $S$ need not be globally compact to admit concentration optimizers, it suffices that compactness effectively interacts with the finite phase-space region $\Omega$. 

On the other hand, there are operators which fail to sufficiently smooth the Wigner distribution, at least in the sense of a positive answer to Problem~\eqref{Problem}. In Section~\ref{sec:Negative} we provide two classes of counterexamples: We show that the quotient is bounded but never attained for the negative compact perturbation of the identity
\[ S=\mathrm{Id}-\varphi_0\otimes\varphi_0, \qquad \varphi_0(y)=2^{d/4}e^{-\pi |y|^2}, \]
and also for every nontrivial time-frequency shift $S=\pi(z_0)$, $z_0\ne0$. Together with the previously established results for the Wigner and Born--Jordan distributions~\cite{Stra-Svela-Trapasso-2025}, this shows that the situation is significantly more challenging for noncompact operators. 

While a full classification of smoothing and escape mechanisms in the Cohen class for $S \in \cB(L^2(\rd))$ remains an open problem (if at all possible), we contribute in this direction with an analysis of the structural properties behind the dichotomy. A key quantity in our analysis is the essential concentration value: 
\[ \Lambda^{\mathrm{ess}}_{p,\Omega}(S) \coloneqq \sup\left\{ \limsup_{n\to\infty}\|Q_Sf_n\|_{L^p(\Omega)}: \|f_n\|_{L^2}=1,\ f_n\rightharpoonup0 \right\}. \]
Roughly speaking, this quantity measures the amount of concentration that can remain in $\Omega$ while the states escape weakly, and thus provides a sufficient criterion for both continuity and attainment. To be specific, in Proposition~\ref{Prop:StrictGap} we show that \(\Lambda^{\mathrm{ess}}_{p,\Omega}(S)<\Lambda_{p,\Omega}(S)\) is sufficient for an optimizer to exist, while in Proposition~\ref{prop:wtn} we show that \(\Lambda^{\mathrm{ess}}_{p,\Omega}(S)=0\) is equivalent to a weakly continuous concentration functional. In spite of the general picture, these results have interesting concrete applications: For instance, in Appendix \ref{app-wig} we use them to rule out exotic optimizers (cf.\ \cite{Stra-Svela-Trapasso-2025} for context) for the Wigner concentration problem over sufficiently small or large phase space balls. 

A second structural contribution concerns the boundary between compactness and weak continuity: We prove that $\cB_{\mathrm{pt}}(p,\Omega)=\cK$, where $\cB_{\mathrm{pt}}(p,\Omega)$ consists of windows for which $Q_Sf_n(z)\to Q_Sf(z)$ pointwise under weak convergence $f_n\rightharpoonup f$. We also show that $\cB_{\mathrm{wc}}(p,\Omega)\cap \mathcal C_{\mathrm u}^\alpha=\cK$, where $\mathcal C_{\mathrm u}^\alpha$ is the class of operators uniformly continuous under phase-space shifts. As a result, any noncompact weakly continuous window must be singular from the QHA viewpoint: weak continuity cannot be detected pointwise and cannot come from uniform shift regularity. The Born--Jordan distribution provides one such singular mechanism, already implicitly exploited in \cite{Stra-Svela-Trapasso-2025,Stra-Svela-Trapasso-2026}. In dimension $d=1$, its window admits the averaged representation
\[ S_{\mathrm{BJ}}=\int_{\mathbb R}\frac{1}{2\cosh(s/2)}D_sP\dd s =\frac{\pi}{\cosh(\pi K)}P, \]
where $D_s$ is the dilation group and $K$ its self-adjoint generator. This representation explains Born--Jordan weak continuity as an averaging effect over metaplectic squeezes. We extend this mechanism to nonatomic averages of $\tau$-Wigner distributions and to a broad class of squeeze averages $\int h(s)D_sP\dd s$ with $h\in L^1(\mathbb R)$, obtaining noncompact examples with weakly continuous concentration functional. In the same spirit, we also show a completely different noncompact mechanism based on positive diagonal operators whose localized averages are compact, upgrading the previous remark about sufficiency of local compactness.

Phase space representations are also extremely useful tools for the analysis of operators, see for instance~\cite{KS07,Schupp,Husimi2,OpSTFT}. One of the most popular notion in this connection is the classical Husimi function: 
\[H_T(z)=\langle T\pi(z)\phi_0,\pi(z)\phi_0\rangle = T\star \widecheck{(\phi_0\otimes\phi_0)}(z),\]
showing that phase space representations of operators can also be studied through QHA. This angle also emphasizes the well-known connection between linear phase space representations of operators and quadratic phase space representations of functions, see for instance~\cite{FrankNicolaTilli,Lieb-Solovej}. With these observations in mind, we will therefore consider an analogue of Problem~\ref{Problem} for a linear representation of operators, namely \(T\star\widecheck{S}\), which should be thought of as a generalization of the Husimi function, and the quadratic self-representation \(\tilde{T}=T\star\widecheck{T}\). The representations \(T\star \widecheck{S}\) have previously been studied by Klauder and Skagerstam~\cite{KS07} and Luef and Skrettingland~\cite{LSBerezin}. Using an important QHA result of Luef and Skrettingland~\cite{LS20} we will show that, for generalized Husimi functions $T\star\widecheck S$, optimizing over Hilbert--Schmidt operators reduces, via Weyl symbols, to a fixed-window convolution problem. By contrast, the total correlation $T\star\widecheck T$ reduces to a genuine autocorrelation problem for Weyl symbols. Our main result is the following: 
\begin{prop}\label{Prop:HSExistence}
    Let \(p\in[1,\infty)\), \(\Omega\subset \R^{2d}\) be measurable with \(0<|\Omega|<\infty\), and let \(S\in \S^2\). Then, the supremum \begin{align*}
        \sup_{T\in\S^2\setminus\{0\}}\frac{\left(\int_{\Omega}|T\star\widecheck{S}(z)|^p \dd{z}\right)^{1/p} }{\|T\|_{\S^2}}
\end{align*}
    is attained, and the corresponding concentration functional is weakly continuous. Moreover, 
    \begin{align*}
        \sup_{T\in\S^2\setminus\{0\}}\frac{\left(\int_{\Omega}|T\star \widecheck{T}(z)|^p \dd{z}\right)^{1/p} }{\|T\|_{\S^2}^2} = |\Omega|^{1/p},
\end{align*}
but this supremum is not attained.
\end{prop}
We will also optimize over the important class of \textit{density operators} (i.e., positive operators of trace $1$), where the optimal value collapses exactly to the original Cohen-class one. To be more precise, this means that the extremal information of the generalized Husimi representation over density operators is already contained in the rank-one Cohen-class problem: 
\begin{prop}\label{DensityOptimization}
Let \(p\in[1,\infty)\), \(\Omega\subset \R^{2d}\) be measurable with \(0<|\Omega|<\infty\), and let \(S\in \mathcal{B}(L^2(\rd))\). Then
\begin{align*}
    \sup_{T\in \D}\left(\int_{\Omega}|T\star\widecheck{S}(z)|^p \dd{z}\right)^{1/p} =\sup_{f\in L^2(\R^d)\setminus\{0\}}\frac{\left(\int_{\Omega}|Q_Sf(z)|^p \dd{z}\right)^{1/p} }{\|f\|_{L^2}^2}=\Lambda_{p,\Omega}(S).
\end{align*}
If the right-hand supremum is attained (e.g., if \(S\) is compact), then the density-operator supremum is attained at a rank-one operator \(f\otimes f\).
\end{prop}
In the same spirit, we also study concentration problems for more general polarized Cohen class on double phase space, in the context of quantum time-frequency analysis \cite{QTFA}. The concentration problem is then reduced to the usual STFT concentration problem for Weyl symbols, giving existence of optimizers by transferring known results. 

Moreover, in Appendix \ref{app-eucl} we briefly investigate the robustness of our results beyond the Heisenberg representation, in the spirit of coorbit theory and non-Euclidean QHA \cite{BBLS22,fulsche}. It turns out that, for a fixed analyzing vector in a strongly continuous unitary representation of a locally compact group, concentration has optimizers by the same compactness mechanism as the fixed-window Husimi problem. On the other hand, the concentration problem for the diagonal voice transform $f\mapsto\langle f,\rho(\cdot)f\rangle$ (and also for the corresponding operator autocorrelation problem) has an inherent representation-theoretic nature. We explicitly investigate the concentration problem for the affine wavelet representation, for which we obtain a negative answer in terms of existence of concentration optimizers --- in contrast with the fixed-window hyperbolic Faber--Krahn theory in \cite{RamosTilli}. The shearlet transform is covered as well. 

The paper is organized as follows. Section~\ref{sec:Prelim} recalls the required material from time-frequency analysis and QHA. In Section~\ref{Sec:direct} we prove the positive existence results, including compact windows, positive perturbations of the identity, locally compact averages and the $L^\infty$ numerical-radius characterization. Section~\ref{sec:Negative} is concerned with nonattainment examples. The classification based on the essential concentration is developed in Section~\ref{sec-structure} along with the compactness barrier and noncompact weak-continuity mechanisms. Finally, Section~\ref{sec:operators} treats concentration problems for operator phase-space representations. The note includes two appendices: in Appendix~\ref{app-wig} we record strict-gap criteria for the Wigner window, while Appendix~\ref{app-eucl} explores concentration problems for coorbit transforms and diagonal affine-wavelet correlations.  

\section{Preliminaries}\label{sec:Prelim}
\subsection{Notation}
The identity operator on \(L^2(\R^d)\) is denoted by \(\mathrm{Id}\). The operator norm is denoted by \(\|\cdot\|_{\cB}\). The Schwartz functions on \(\R^{d}\) are denoted by \(\mathscr{S}(\R^{d})\), and their dual space of temperate distributions by \(\mathscr{S}'(\R^d)\). If \(f_n\) converges weakly to \(f\), we write \(f_n\rightharpoonup f\).

The spaces of (linear) bounded and compact operators on a complex Hilbert space $\cH$ are denoted by $\cB(\cH)$ and $\cK(\cH)$ respectively. In the case where $\cH=L^2(\rd)$ we usually omit the dependence on $\cH$ and write just $\cB$ and $\cK$. 

The numerical radius of an operator $S \in \cB$ is defined by 
\[ w(S) \coloneqq \sup_{\|g\|_{L^2}=1}|\langle Sg,g\rangle|. \]

\subsection{Time--frequency distributions}
We say that a quadratic time-frequency representation \(Q \colon L^2(\R^d)\to \Schw'(\rdd)\) is \textit{covariant} with respect to time-frequency shifts if \[Q(\pi(z)f)=T_zQ(f), \qquad \forall z\in \R^{2d},\] where we set \(\pi(z)f(y)=\pi(x,\xi)f(y)=e^{2\pi i \xi\cdot y}f(y-x)\) for \(z\in \R^{2d}\) and \(f\in L^2(\R^{d})\), while \(T_z\) denotes the translation by \(z\) on $\rdd$.  

We already defined the Cohen class in the Introduction. Among the members of this family, the spectrogram deserves a special mention. Given a fixed window function \(g\in L^2(\R^d)\) the short-time Fourier transform (STFT) of \(f\) with respect to \(g\) is \begin{align*}
    V_gf(z)=\int_{\R^d}e^{-2\pi i \xi\cdot y}f(y)\overline{g(y-x)} \dd{y}=\langle f,\pi(z)g\rangle.
\end{align*}
The spectrogram is the squared modulus of the STFT: \(|V_gf(z)|^2\). Moreover, using Moyal's identity for the Wigner distribution~\cite[Proposition 4.3.2]{Grochenig}, it is recognized as a Cohen's class distribution: \(|V_gf|^2=Wf*Wg\). Compared to the Wigner distribution, the spectrogram's advantage is that it is always positive. On the other hand, the spectrogram does not satisfy the marginal properties, and is also highly dependent on the choice of window function \(g\). Other Cohen's class distributions have their own advantages and disadvantages, we direct the reader to \cite[Chapter 4]{HlawatschAuger} for an overview. 

While not being central to our problems, we will frequently allude to the \(\tau\)-Wigner and Born--Jordan distributions as illustrating examples. They are Cohen's class distributions closely related to the Wigner distribution. Given \(f\in L^2(\R^d)\) and \(\tau\in [0,1]\) we define\begin{align*}
    W_{\tau}f(z)=\int_{\R^d}e^{-2\pi i \xi \cdot y} f\left(x+\tau y\right)\overline{f\left(x-(1-\tau)y\right)} \dd{y},\quad \BJ f(z)=\int_0^1W_{\tau}f(z) \dd{\tau}.
\end{align*}
The concentration problem~\eqref{Problem} for \(W_\tau\) and \(\BJ\) was studied in~\cite{Stra-Svela-Trapasso-2025,Stra-Svela-Trapasso-2026}.

Let us also briefly mention the ambiguity transform. It is defined by\begin{align*}
    Af(z)=\int_{\R^d} e^{-2\pi i \xi \cdot y}f\left(y+\frac{x}{2}\right)\overline{f\left(y-\frac{x}{2}\right)} \dd{y}=e^{\pi i x\cdot \xi} V_ff(z)
\end{align*}
for \(f\in L^2(\R^d)\). It is related to the Wigner distribution via the symplectic Fourier transform, (formally) defined by \[\mathcal{F}_\sigma F(z) \coloneqq \int_{\R^{2d}}F(u)e^{-2\pi i [u,z]} \dd{u}, \qquad F \in \Schw'(\rdd),\] where \([u,z]\) denotes the symplectic inner product of \(z\) and \(u\), that is 
\[ [u,z] \coloneqq x_u\cdot\xi_z-\xi_u\cdot x_z=u\cdot Jz, \qquad J\coloneqq \begin{pmatrix}
    O_d & I_d \\
    -I_d & O_d
\end{pmatrix}\in \R^{2d\times 2d}. \] Indeed, it is then easy to prove the following relation:
\begin{align*}
    \mathcal{F}_{\sigma}Wf(z)=Af(z),
\end{align*}
which also shows that the ambiguity function fails to be covariant, and thus to belong to the Cohen class. An important property of the ambiguity function is the so-called \textit{radar correlation estimate}~\cite[Lemma 4.2.1]{Grochenig}: For any \(f\in L^2(\R^d)\setminus \{0\}\) and \(z\in \R^{2d}\) we have \begin{align}\label{radar}
    |Af(z)|\leq |Af(0)|=\|f\|_{L^2}^2
\end{align}
with equality if and only if \(z=0\).  This will be of use in the study of Cohen's class.

\subsection{Quantum harmonic analysis}

Quantum harmonic analysis, introduced by Werner in ~\cite{Werner}, is an extension of harmonic analysis to spaces of operators. The standard convolution of functions is accompanied by a convolution between operators, and one between functions and operators. Several objects from time-frequency analysis arise naturally as convolutions or Fourier transforms in QHA, see for instance~\cite{CohenClass}. In this section, we will recall the results from QHA which we rely on in the rest of the text.

The starting point of quantum harmonic analysis is the following two actions on operators: The \textit{operator shift}, which for a \(z\in\R^{2d}\) and \(A\in \cB\) is defined by\begin{align*}
    \alpha_z(A)=\pi(z)A\pi(z)^*,
\end{align*}
and the \textit{operator reflection}, defined by \begin{align*}
    \widecheck{A}=PAP,
\end{align*}
where \(P\) is the parity operator, defined by \(Pf(t)=f(-t)\).  With these definitions in hand, Werner defined the convolution \(T\star S\) between two operators as the function
\begin{align}\label{conv}
    T\star S(z) \coloneqq \tr\big(T\alpha_z(\widecheck{S})\big).
\end{align}
Recall here that the trace of an operator \(T\) on \(L^2(\R^d)\) is given by the series \(\sum_{n=0}^\infty \langle T e_n,e_n\rangle\), where \(\{e_n\}_{n=0}^\infty\) is any orthonormal basis of \(L^2(\R^d)\). While formally sound, operator convolution is initially only well-defined for operators in the trace class, that is, the operators \begin{align*}
    \S^1=\{S\in \mathcal{K} : \tr(|S|)<\infty\}.
\end{align*}
In this case the right-hand side of \eqref{conv} is independent of the choice of orthonormal basis, and thus well-defined. The same formula is also meaningful whenever one factor is trace class and the other is bounded, since \(T\alpha_z(\widecheck S)\in\S^1\) for \(T\in\S^1\) and \(S\in\cB\). In fact, one can even verify that for \(S,T\in \S^1\), \(T\star S\in L^1(\R^{2d})\). 

One of the most important class of elements in $\S^1$ are finite rank operators, in particular \textit{rank-one operators}: Given $f,g \in L^2(\rd)$, $f\otimes g \in \S^1$ is defined by 
\[  (f\otimes g) h=\langle h,g\rangle f, \quad h \in L^2(\rd). \] 
Moreover, $\S^1$ is just a special case of the Schatten classes 
\begin{align*}
    \S^p=\{S\in \mathcal{K} : \tr(|S|^p)<\infty\}, \qquad 1\le p <\infty,
\end{align*}
which are Banach spaces of operators with the norm \(\|S\|_{\S^p}=\left(\tr(|S|^p)\right)^{1/p}\). Alternatively, we may characterize the Schatten classes in terms of singular values: A compact operator \(S\) belongs to \(\S^p\) if and only if the singular values \(\sigma(n)\) of \(S\) belong to \(\ell^p\), and the Schatten norm equals the \(\ell^p\) norm of the sequence of singular values. In QHA, the Schatten classes are the operator analogues of the \(L^p\)-spaces, and they obey a similar duality relation, namely \((S^p)^*=S^{q},\) where \(\frac{1}{p}+\frac{1}{q}=1\) for \(p\in (1,\infty)\). \(\S^2\) is therefore a Hilbert space, namely the space of Hilbert--Schmidt operators. For the endpoint cases we have \((\S^1)^*=\cB\) and \(\mathcal{K}^*=\S^1\), thus leading us to conveniently set \(\S^\infty=\cB\). 
The duality is concretely given by \begin{align*}
    \langle T,S \rangle=\tr(TS^*).
\end{align*}
Along the lines of the same heuristics, we note that operator convolution extends to Schatten classes similarly to how regular convolution extends to the \(L^p\)-spaces, that is via the following analogue of Young's inequality~\cite{Werner}: 
\begin{align*}
    \|T\star S\|_{L^r}\leq \|T\|_{\S^p}\|S\|_{\S^q}, \qquad \frac{1}{p}+\frac{1}{q}=1+\frac{1}{r}. 
\end{align*}
Let us emphasize that operator convolution is commutative and the convolution of two positive operators results in a positive function.
There is also an adjoint notion of function-operator convolution: For \(F\in L^1(\rdd)\) and \(S\in\cB\), we set
\[ F\star S\coloneqq \int_{\rdd} F(z)\alpha_z(S) \dd{z}, \]
to be interpreted as a weak operator-valued integral that is,
\[ \langle (F\star S)u,v\rangle = \int_{\rdd}F(z)\langle \alpha_z(S)u,v\rangle \dd{z}, \qquad u,v \in L^2(\rd). \]
As before, we get refined estimates via Werner's second Young inequality:
\begin{align*}
    \|F\star S\|_{\S^r}\leq \|F\|_{L^p}\|S\|_{\S^q}, \qquad \frac{1}{p}+\frac{1}{q}=1+\frac{1}{r}.
\end{align*}
In particular, we have $\|F\star S\|_{\cB}\le \|F\|_{L^1}\|S\|_{\cB}$. 

Convolution of an operator with a characteristic function of a set $\Omega \subset \rdd$ will play a distinguished role. These are the so-called \emph{mixed-state localization operators}~\cite{CohenClass}: 
\[ H_{\Omega,S}\coloneqq \chi_\Omega\star S = \int_\Omega \alpha_z(S) \dd{z}. \]

QHA also provides a notion of Fourier transform for operators, the \textit{Fourier-Wigner transform}. For \(S\in \S^1\), the Fourier-Wigner transform is the bounded function defined by \begin{align*}
    \mathcal{F}_W(S)(z) \coloneqq e^{-\pi i x \cdot \xi} \tr(\pi(-z)S), \qquad z\in \R^{2d}.
\end{align*}
The Fourier-Wigner transform shares indeed many of its properties with the classical Fourier transform, for instance it decouples Werner's convolutions:\begin{align*}
    \mathcal{F}_{\sigma}(T\star S)=\mathcal{F}_W(T)\mathcal{F}_W(S),\qquad \mathcal{F}_{W}(F\star S)=\mathcal{F}_\sigma(F)\mathcal{F}_W(S).
\end{align*}
It also allows us to illustrate the connection between QHA and time-frequency analysis, as the following identity shows.\begin{lemma}[\cite{Bible1}, Lemma 6.1]
    For every \(f\in L^2(\R^d)\), the Fourier-Wigner transform of \(f\otimes f\) is the ambiguity function:\begin{align*}
        \mathcal{F}_W(f\otimes f)(z)=Af(z), \qquad z \in \rdd.
    \end{align*}
\end{lemma}

Since QHA studies the interactions between functions on \(\R^{2d}\) and operators on \(L^2(\R^d)\), it is closely related to the concept of quantization, and in particular the Weyl calculus. Given a temperate distribution \(F\in \mathscr{S}'(\R^{2d})\), its Weyl transform \(L_F \colon \sS(\rd) \to \sS'(\rd)\) is defined weakly by
\begin{align*}
    \langle L_F\phi,\psi\rangle=\langle F, W(\psi,\phi)\rangle,
\end{align*} for \(\phi,\psi\in \mathscr{S}(\R^{d})\). Conversely, by the Schwartz kernel theorem it follows that every continuous linear operator $S \colon \sS(\rd) \to \sS'(\rd)$ can be viewed as the Weyl quantization of a suitable symbol \(a_S \in \Schw'(\rdd)\), that is $S=L_{a_S}$.  

We denote by \(\mathfrak{S}\) all operators on \(L^2(\R^d)\) with Weyl symbol in \(\mathscr{S}(\R^{2d})\) and the operators with symbol in \(\mathscr{S}'(\R^{2d})\) by \(\mathfrak{S}'\). By means of a duality argument it is possible to extend operator convolution \(T\star S\) and the Fourier-Wigner transform \(\mathcal{F}_W(S)\) to the cases \(S\in \mathfrak{S}'\) and \(T\in \mathfrak{S}\), the resulting objects being temperate distributions on \(\R^{2d}\) --- see~\cite{SchwartzOps} for more details. We also recall the following result from \cite{SchwartzOps}, which is the QHA analogue of the Fourier inversion theorem. 
\begin{prop}
    Given \(S\in \mathfrak{S}'\), let \(a_S\in \mathscr{S}'(\R^{2d})\) denote its Weyl symbol. Then $        \mathcal{F}_{\sigma}\mathcal{F}_W(S)=a_S$, or equivalently $\mathcal{F}_W^{-1}=L\circ\mathcal{F}_\sigma$. 
\end{prop}

Our motivation for introducing operator convolutions is that they provide a different approach in the study of Cohen's class distributions. As mentioned in the Introduction, the following observation from~\cite{CohenClass} is crucial in this regard.  
\begin{prop}[{\cite[Proposition 7.1]{CohenClass}}]
Consider \(\sigma\in \mathscr{S}'(\R^{2d})\) and \(f\in \mathscr{S}(\R^d)\). The Cohen class distribution \(Q_{\sigma}f\) coincides with the operator convolution \begin{align*}
    Q_{\sigma}f=(f\otimes f)\star L_{\sigma}.
\end{align*}
Conversely, every operator \(S\in \mathfrak{S}'\) defines a Cohen's class distribution via \begin{align*}
    Q_Sf=(f\otimes f)\star \widecheck{S}.
\end{align*}
\end{prop}
Consequently, we can view Cohen's class distributions as special cases of operator convolutions. For the sake of completeness, let us express some popular Cohen's class distributions as operator convolutions.
\begin{lemma}\label{Calculations}
    The (generalized) spectrograms, the (\(\tau\))-Wigner distributions, and the Born-Jordan distribution take the following form when written as operator convolutions:
    \begin{itemize}
        \item \(|V_gf|^2=(f\otimes f)\star (\widecheck{g}\otimes\widecheck{g})\).
        \item \(V_gf\overline{V_hf}=(f\otimes f)\star (\widecheck{h}\otimes\widecheck{g})\).
        \item \(Wf=(f\otimes f)\star 2^dP\).
        \item For $0<\tau<1$, \(W_{\tau}f=(f\otimes f)\star \widecheck{S_{\tau}},\) where \(S_{\tau}f(y)=\frac{1}{(1-\tau)^d}f\left(\frac{\tau}{\tau-1}\cdot y\right)\).
        \item \(\BJ f=(f\otimes f)\star \widecheck{S}_{\mathrm{BJ}}, \) where \(S_{\mathrm{BJ}}=\mathcal{F}_W^{-1}(\mathrm{sinc}(\pi x\cdot \xi))\).
    \end{itemize}
\end{lemma}
\begin{proof}
    For the first three items, see~\cite[Lemma 3.1, Example 7.1]{CohenClass}, for the last two, see~\cite[Section 6]{Luef-Skrettingland-2021}. Note for future reference that we will sometimes denote \(2^dP\) by \(S_{\mathrm{W}}=S_{1/2}\).
\end{proof}

Since the rank-one operators are in a sense the most fundamental ones, the above result justifies thinking of the (generalized) spectrograms as the most fundamental elements of Cohen's class, at least in the QHA approach.

We conclude this section with a couple of technical results that will be repeatedly used later. 
\begin{lemma}\label{lem:nonzero-window-nonvanishing}
Consider \(S\in \cB \smo\) and a set \(\Omega\subset\rdd\) of positive measure. There exists
\(u\in L^2(\rd)\), \(\|u\|_{L^2}=1\), such that \(|Q_Su|>0\) on a subset of \(\Omega\) of positive measure. In particular, \(\|Q_Su\|_{L^p(\Omega)}>0\) for every \(1\le p\le\infty\).
\end{lemma}

\begin{proof}
Since \(S\ne0\) and the Hilbert space is complex, there exists
\(v\in L^2(\rd)\) with \(\|v\|_{L^2}=1\) such that \(\langle Sv,v\rangle\ne0\) --- otherwise \(\langle Sf,f\rangle=0\) for every \(f\), so by polarization \(\langle Sf,g\rangle=0\) for every \(f,g\) would force \(S=0\).

Let \(z_0\) be a Lebesgue density point of \(\Omega\) and set \(u\coloneqq \pi(z_0)v\). Then \(\|u\|_{L^2}=1\) and 
\[Q_Su(z_0)=\langle \alpha_{z_0}(S)u,u\rangle=\langle \pi(z_0)S\pi(z_0)^*\pi(z_0)v,\pi(z_0)v\rangle=\langle Sv,v\rangle\ne0.\]
The map \(z\mapsto Q_Su(z)\) is continuous, because \(z\mapsto\pi(z)\) is strongly continuous and \(S\) is bounded, therefore there are a neighborhood \(U\) of \(z_0\) and a constant \(c>0\) such that \(|\Omega\cap U|>0\) and \(|Q_Su(z)|\ge c\) for all \(z\in U\). The claim then follows. 
\end{proof}

\begin{lemma}\label{lem-mixed}
For every \(1\le p<\infty\), \(S\in\cB\), \(f\in L^2(\rd)\), and \(\Omega\subset\rdd\) with finite measure,
\[ r_n\rightharpoonup0 \text{ in } L^2(\rd) \implies \|Q_S(r_n,f)\|_{L^p_z(\Omega)}\to0, \quad \|Q_S(f,r_n)\|_{L^p_z(\Omega)}\to0.\]
\end{lemma}

\begin{proof}
For every $z \in \rdd$ we have $\langle \alpha_z(S)r_n,f\rangle = \langle r_n,\alpha_z(S)^*f\rangle$, so weak convergence gives $\langle \alpha_z(S)r_n,f\rangle\to0$. Similarly, we have $\langle \alpha_z(S)f,r_n\rangle = \overline{\langle r_n,\alpha_z(S)f\rangle} \to0$, so both mixed terms vanish pointwise on \(\Omega\). 

Since \(r_n\rightharpoonup0\), there exists \(M>0\) such that $\|r_n\|_{L^2}\le M$ for all $n$. We also have
\[ |\langle \alpha_z(S)r_n,f\rangle| \le \|\alpha_z(S)\|_{\cB}\|r_n\|_{L^2}\|f\|_{L^2}
= \|S\|_{\cB}\|r_n\|_{L^2}\|f\|_{L^2} \le M\|S\|_{\cB}\|f\|_{L^2}, \] and similarly \( |\langle \alpha_z(S)f,r_n\rangle| \le M\|S\|_{\cB}\|f\|_{L^2}\). It is clear that the constant function $z \mapsto M\|S\|_{\cB}\|f\|_{L^2}$ is $p$-integrable over $\Omega$ since $|\Omega|<\infty$, and the claim thus follows by dominated convergence. 
\end{proof}

\section{Positive existence results}\label{Sec:direct}

\subsection{Compact windows}
By the singular value theorem, every compact operator \(S \in \cK\) is a convergent sum of rank-one operators. As a consequence, any conclusion we reach in the rank-one case can be transplanted to compact operators. It is easily seen (see \cite[Proposition 5.1]{Nicola-Romero-Trapasso-2022} for details) that the concentration functional \textit{is} weakly continuous when \(S=g\otimes g\) for any non-trivial $g \in L^2(\rd)$. By using the singular value decomposition we can thus show weak continuity of the concentration functional for \textit{any} compact operator window \(S\). The key idea is to lift the question of continuity to the rank-one operators. The following lemma clarifies how lifting affects continuity.
\begin{lemma}\label{ConvergenceLemma}
    Consider \(f,f_n\in L^2(\R^d)\) for all \(n\in\N\). If \(f_n\rightharpoonup f\) in \(L^2(\R^d)\) then 
    \[
	f_n\otimes f_n \rightharpoonup^* f\otimes f
	\qquad\text{in} \quad \cS^1.
	\]
\end{lemma}

\begin{proof} 
We need to prove that, for an arbitrary $T \in \cK$,
	\[
	\tr\bigl((f_n\otimes f_n)T^*\bigr) 	\to \tr\bigl((f\otimes f)T^*\bigr).
	\]
    A straightforward computation shows that 
    \[
    \tr\big((f_n\otimes f_n)T^*\big)=\langle T^*f_n,f_n\rangle = \langle f_n,Tf_n\rangle,
    \]
    hence it suffices to prove that \(\langle T^*f_n,f_n\rangle\to \langle T^*f,f\rangle\). Since \(f_n\rightharpoonup f\) we have that $(f_n)_n$ is bounded in \(L^2(\R^d)\), and also that \(T^*f_n\to T^*f\) strongly in \(L^2(\R^d)\) due to compactness of \(T^*\). As a result, we obtain
    \begin{align*}
        \big|\langle T^*f_n,f_n\rangle-\langle T^*f,f\rangle\big|
        &\leq \big|\langle T^*(f_n-f),f_n\rangle\big|+\big|\langle T^*f,f_n-f\rangle\big|\\
        &\leq \|T^*f_n-T^*f\|_{L^2}\,\Big(\sup_n\|f_n\|_{L^2}\Big)+\big|\langle T^*f,f_n-f\rangle\big|\to 0,
    \end{align*} that is the claim. 

\end{proof}

\begin{remark} Note that the rank-one operator \(f\otimes f\) is in \(\S^p\) for any \(f\in L^2(\R^d)\) and \(p\in [1,\infty]\). One can thus use the argument of Lemma \ref{ConvergenceLemma} to show that \(f_n\otimes f_n\) converges to \(f\otimes f\) weakly in \(\S^p\) for any \(p\in (1,\infty)\), weakly in \(\mathcal{K}\), and weak-* in \(\cB\). Although Lemma \ref{ConvergenceLemma} ensures weak-* convergence in \(\S^1\), \(f_n\otimes f_n\) does \textit{not} converge weakly to \(f\otimes f\) in \(\S^1\). To see this, recall that \(\left(\S^{1}\right)^*=\cB\). If we pair with \(T=\mathrm{Id}\) we get\begin{align*}
    \tr\left((f_n\otimes f_n) \;\mathrm{Id}\right)=  \tr\left(f_n\otimes f_n\right)=\|f_n\|_{L^2}^2,
\end{align*} so weak convergence in \(\S^1\) ultimately requires norm convergence in \(L^2\). 
\end{remark}

The result in Lemma \ref{ConvergenceLemma} plays a critical role in proving our first main result. 

\begin{proof}[Proof of Theorem~\ref{CompactExistence}]
    Let us first show weak continuity of \(f\mapsto Q_Sf(z)=(f\otimes f)\star\widecheck{S}(z)=\tr\left((f\otimes f)\alpha_z(S)\right)\). To that end, fix \(z\in\R^{2d}\) and note that if \(f_n\rightharpoonup f\) then \(f_n\otimes f_n\rightharpoonup^* f\otimes f\) in \(\cS^1\) by Lemma \ref{ConvergenceLemma}. Since \(S\in \mathcal{K}\), so is \(\alpha_z(S)\). Consequently,  for all \(z\in\R^{2d}\) we have 
    \begin{align*}
        \lim_{n\rightarrow\infty} Q_Sf_n(z)=\lim_{n\rightarrow\infty}\tr\left((f_n\otimes f_n)\alpha_z(S)\right)=\tr\left((f\otimes f)\alpha_z(S)\right)=Q_Sf(z).
    \end{align*}
   Note furthermore that for any \(f\in L^2(\R^d)\) we have the pointwise estimate 
   \begin{align*}
        |Q_Sf(z)|\leq \|Q_Sf\|_{L^\infty}\leq \|f\otimes f\|_{\S^1}\|S\|_{\cB}=\|f\|_{L^2}^2\|S\|_{\cB},
    \end{align*}
    where the second inequality is Young's inequality for operator convolutions, and therefore for \(\{f_n\}\) the uniform bound
    \[ \abs{Q_S f_n(z)} \le M^2 \norm{S}_\cB, \qquad M=\sup_n \norm{f_n}_{L^2}<\infty, \quad z \in \rdd.\]  We can thus argue by dominated convergence to conclude:
    \begin{align*}
        \lim_{n\rightarrow\infty} \left(\int_{\Omega}|Q_Sf_n(z)|^p \dd{z}\right)^{1/p} &=\left(\lim_{n\to\infty}\int_{\Omega}|Q_Sf_n(z)|^p \dd{z}\right)^{1/p} \\&=\left(\int_{\Omega}\lim_{n\rightarrow\infty} |Q_Sf_n(z)|^p \dd{z}\right)^{1/p} =\left(\int_{\Omega} |Q_Sf(z)|^p \dd{z}\right)^{1/p} . 
    \end{align*}
    Existence of maximizers now follows by the direct method. Indeed, choose a normalized maximizing sequence \(\|f_n\|_{L^2}=1\) and pass to a weakly convergent subsequence \(f_{n_k}\rightharpoonup f\). We can assume \(S\neq0\), otherwise the claim is trivial, and then one has \(\Lambda_{p,\Omega}(S)>0\) by Lemma \ref{lem:nonzero-window-nonvanishing}. Weak continuity then gives \(J_S(f)=\Lambda_{p,\Omega}(S)\), so \(f\neq0\). If \(\|f\|_{L^2}<1\), by homogeneity we have
    \[ J_S(f/\|f\|_{L^2})=J_S(f)/\|f\|_{L^2}^2>\Lambda_{p,\Omega}(S), \]
    a contradiction. Thus \(\|f\|_{L^2}=1\), and weak convergence together with convergence of norms gives \(f_{n_k}\to f\) strongly in \(L^2\).
    \end{proof}

\begin{remark}
    It is important to emphasize here that it would not be possible to prove the above result if one started from the definition \(Q_\sigma f=Wf*\sigma\). Indeed, as shown in~\cite[Proposition 3.5]{Stra-Svela-Trapasso-2025}, the concentration functional \(f\mapsto \|Wf\|_{L^p(\Omega)}\) for the Wigner distribution is not even (weakly upper) semicontinuous. On the other hand, the weak continuity of the elementary tensor windows (hence, spectrograms) is key to the proof, and it is easy to realize that the same conclusions hold for generalized concentration problems where $\norm{Q_Sf}_{L^p(\Omega)}$ is replaced with $\norm{m \, Q_Sf}_{L^p}$ for an integrable phase-space weight $m \in L^p(\rdd)$. 
\end{remark}

\begin{remark} \label{rem-eucl} The proof of Theorem~\ref{CompactExistence} moves along representation-theoretic arguments, with no special role of the underlying Euclidean structure. It is easy to realize that the same argument extends indeed to QHA over abelian phase spaces \cite{fulsche}: The Euclidean phase space is replaced by a locally compact abelian phase space \(\Xi\) equipped with Haar measure, and the operator shifts are built from a projective unitary representation \(U \colon \Xi\to\mathcal U(\cH)\), namely \(\alpha_\xi(S)=U_\xi S U_\xi^*\) --- consider for concreteness the case \(\Xi=G\times\widehat G\), where \(G\) is a locally compact abelian group. Related extensions appear in Appendix \ref{app-eucl} below, and further developments can be found in \cite{gro_nilpotent}.
\end{remark}

Theorem \ref{CompactExistence} confirms the existence of an optimizer for all Cohen's classes with a compact operator window. In addition to the spectrograms \(|V_gf(z)|^2\), for which this existence was already known, there are two other important Cohen's classes with compact window. 
\begin{example}
    The first class is the Gaussian smoothed Wigner distributions \(Q_\sigma f=Wf*g_{\Sigma}\), where \(g_{\Sigma}(z)=(2\pi)^{-d}\sqrt{\det \Sigma^{-1}}\mathrm{exp}\left({-\frac{1}{2}\langle \Sigma^{-1}z,z\rangle}\right)\) is a \(2d\)-dimensional Gaussian with covariance matrix \(\Sigma\). Investigations on these Cohen's class distributions have mostly been concerned with positivity, see~\cite{Grochenig,deBruijn,Yvon,CdGN2019}. When the matrix \(\Sigma+\frac{i}{4\pi}J\) is positive semidefinite the distribution is non-negative for all \(f\in L^2(\R^d)\). 
    
    In the operator formulation, these Cohen's classes coincide with \(Q_Sf = f\otimes f \star \widecheck{L}_{g_{\Sigma}}\), and the Weyl transform of \(g_{\Sigma}\) is known to be in \(\S^1\) for all \(\Sigma\) satisfying the above condition~\cite{CdGN2019}. As such, Theorem~\ref{CompactExistence} applies. Note that when \(d=1\) and \(\Sigma=\frac{2E+1}{4\pi}\mathrm{Id}\), $S$ corresponds to the thermal states in quantum mechanics: 
    \[ L_{g_\Sigma}=\frac{1}{E+1}\sum_{n=0}^\infty \left(\frac{E}{E+1}\right)^n\phi_n\otimes\phi_n,\] where \(\phi_n\) denotes the \(n\)-th Hermite function.  
\end{example} 

\begin{example}
    Another relevant class is that of smoothed spectrograms, that is \(F*|V_gf|^2\) with \(F\in L^p(\R^{2d}),\; p\in[1,\infty)\). These distributions are connected to convolutional neural networks, and represent the effect of a convolutional layer on a spectrogram, see~\cite{BasicFilter,InsideSpec}. 
    The operator window \(S\) is in this case the function operator convolution $F\star g\otimes g$. By Young's inequality \(\|F\star g\otimes g\|_{S^p}\leq \|F\|_{L^p}\|g\|_{L^2}^2\), and so \(F\star g\otimes g\) is compact and Theorem~\ref{CompactExistence} applies.
\end{example}
\begin{remark}
        Theorem~\ref{CompactExistence} might appear of limited scope at a first glance. Indeed, aside from the spectrograms most of the distinguished Cohen's class distributions historically used in applications (see~\cite[Chapter 4]{HlawatschAuger} or \cite[Chapter 2]{FlandrinTimeFreqScale}) use sinusoidal or chirp like smoothing functions \(\sigma\), resulting in a noncompact operator window \(S\). On the other hand, compact Cohen's class distributions are essentially weighted sums of spectrograms --- more precisely, operator-norm limits of linear combinations of finite-rank windows. Moreover, with the advent of deep learning, we would argue these representations are by far the most natural from the perspective of data-driven time-frequency analysis. 
        In addition to the vast amount of spectrogram-based methods in machine learning~\cite{SpecAugment} and signal processing~\cite{Smith11}, it has also been shown that, starting from spectrogram input data, convolutional neural networks (CNNs) are able to learn the properties of other time-frequency representations~\cite{BasicFilter}. Note that when training on spectrograms the convolutional part of a CNN corresponds to convolving with a compactly supported kernel \(m\in L^1(\R^{2})\), which in the QHA approach represents a change of compact Cohen's class from \(|V_gf(z)|^2\) to \(Q_{m\star(g\otimes g)}f\) --- precisely the representations covered by Theorem~\ref{CompactExistence}. 
\end{remark}

While an optimizer for Problem~\ref{Problem} exists for any compact \(S\), it is highly non-trivial to determine the value of the supremum without any further assumptions. We already mentioned that the best estimate one can get is in general \( \Lambda_{p,\Omega}(S)\leq |\Omega|^{1/p}\|S\|_{\cB}\). However, under the assumption of positivity, Jensen's inequality provides a slight improvement. The following result is a refinement of \cite[Proposition 3.2]{LocalStruc}. 
\begin{prop}[Jensen's inequality for convolution]
    Let \(\Phi\) be a nonnegative, convex and continuous function on \(\R^+\) with \(\Phi(0)=0\), and let \(T\in \S^1\) be positive with \(\|T\|_{\S^1}=1\). If \(S\) is a positive compact operator 
    then for all \(z\in \R^{2d}\)
    \begin{align*}
         \Phi\bigl((T\star S)(z)\bigr)\leq \bigl(T\star \Phi(S)\bigr)(z),
    \end{align*}
    where \(\Phi(S)\) is defined via functional calculus.
\end{prop}
\begin{proof}
    Firstly, note that since \(\Phi(0)=0\), \(\Phi(S)\) will also be compact. This ensures that \(T\star \Phi(S)\) is both bounded (by Young's inequality) and continuous in \(z\) (by~\cite[Proposition 4.6]{Bible1}). Since \(S\) is positive and compact, the same holds for $\widecheck{S}$ and we may perform the spectral decomposition \(\widecheck{S}=\sum_{n=0}^\infty \lambda_n \psi_n\otimes \psi_n\). Fix then \(z\in\R^{2d}\) and define 
    \[ a_n(z) \coloneqq \tr\bigl(T\alpha_z(\psi_n\otimes\psi_n)\bigr)=\langle T\pi(z)\psi_n,\pi(z)\psi_n\rangle. \] 
    It is then clear, by positivity of $T$ and since \((\pi(z)\psi_n)_{n=0}^\infty\) is an orthonormal basis,  that $a_n(z)\geq0$ and $\sum_{n=0}^\infty a_n(z)=\tr(T)=1$, so \((a_n(z))_{n=0}^\infty\) yields a probability distribution on \(\N_0\). As such, since \(\lambda_n\in[0,\|S\|_{\cB}]\) and \(\Phi\) is continuous and convex on this compact interval, Jensen's inequality gives
    \[ \Phi\bigl((T\star S)(z)\bigr) = \Phi\left(\sum_{n=0}^\infty \lambda_n a_n(z)\right) \leq \sum_{n=0}^\infty \Phi(\lambda_n)a_n(z). \]
    To conclude, since $\widecheck S=P S P$ implies $\Phi(\widecheck S)=P\Phi(S)P=\widecheck{\Phi(S)}$, we obtain 
    \[ \widecheck{\Phi(S)} = \sum_{n=0}^\infty \Phi(\lambda_n)\,\psi_n\otimes\psi_n, \] and thus
    \[ (T\star\Phi(S))(z) = \tr\bigl(T\alpha_z(\widecheck{\Phi(S)})\bigr) = \sum_{n=0}^\infty \Phi(\lambda_n)a_n(z). \]
    The claim now follows by comparison. 
    
\end{proof}
If we assume \(\|f\|_{L^2}=1\), then $S\ge 0$ implies $Q_Sf\ge0$, so the special case \(\Phi(x)=x^p\), \(T=f\otimes f\) and \(\widecheck S\) in place of \(S\) of the above inequality yields
\begin{align*}
    \|Q_Sf\|_{L^p(\Omega)}^p &=\int_\Omega(Q_Sf(z))^p \dd{z}\ \leq\int_\Omega Q_{S^p}f(z) \dd{z} =\left\langle H_{\Omega,S^p}f,f\right\rangle, 
 \end{align*} where we recall the definition \( H_{\Omega,S^p}=\chi_\Omega\star S^p=\int_\Omega \alpha_z(S^p) \dd{z}\). Therefore, we obtain 
\[
    \Lambda_{p,\Omega}(S)\leq \|H_{\Omega,S^p}\|_{\cB}^{1/p},
\] which can be sharper than Young's universal estimate, since \(\|H_{\Omega,S^p}\|_{\cB} \leq |\Omega|\|S^p\|_{\cB}\).

\subsection{Positive mechanisms beyond compactness}
It is clear that, although the proof strategy of Theorem \ref{CompactExistence} fails without the assumption \(S \in \cK\), that does not prevent existence of optimizers beyond the compact realm. Let us revisit the example of \(S=\mathrm{Id}\). We have that \(Q_{\mathrm{Id}}f(z)=\|f\|_{L^2}^2\), therefore the supremum \eqref{Problem} is \(|\Omega|^{1/p} \) and any normalized \(f\) is an optimizer --- despite failure of weak continuity in \(f\). 

With this observation in mind, we can extend the results for the compact case to positive compact perturbations of the identity. The crucial point here is that we are able to control the noncompact part of \(S\).

\begin{prop}\label{prop:perturbation}
Let \(S_0\in\cK(L^2(\rd))\) be positive, let \(c\ge0\), and set \(S=c\,\mathrm{Id}+S_0\). Let \(1\le p<\infty\), and let \(\Omega\subset\rdd\) satisfy \(0<|\Omega|<\infty\). If \(\|f_n\|_{L^2}=1\) and \(f_n\rightharpoonup f\), then
\[
    J_S(f_n) \longrightarrow \bigl\|Q_Sf+c(1-\|f\|_{L^2}^2)\bigr\|_{L^p(\Omega)} = \bigl\|c+Q_{S_0}f\bigr\|_{L^p(\Omega)} .
\]
The supremum \(\Lambda_{p,\Omega}(S)\) is therefore attained.
\end{prop}

\begin{proof}
Since \(S_0\) is compact, the compact-window argument gives \( Q_{S_0}f_n\to Q_{S_0}f \) in \(L^p(\Omega)\). Moreover, since \(\|f_n\|_{L^2}=1\), we have \(Q_Sf_n=c+Q_{S_0}f_n\). On the other hand, \(Q_Sf=c\|f\|_{L^2}^2+Q_{S_0}f\), hence 
\[ Q_Sf_n\to c+Q_{S_0}f=Q_Sf+c(1-\|f\|_{L^2}^2)\qquad\text{in }L^p(\Omega),\]
which proves the first part of the claim. 

Let us now prove attainment. If \(S_0=0\), every normalized \(f\) is an optimizer, so we assume \(S_0\neq0\). By Lemma~\ref{lem:nonzero-window-nonvanishing}, there exists \(u\in L^2(\rd)\) with \(\|u\|_{L^2}=1\) such that \(|Q_{S_0}u|>0\) on a subset of \(\Omega\) of positive measure. Since \(S_0\ge0\) we have \(Q_{S_0}u\ge0\), and hence \(Q_{S_0}u>0\) on a subset of \(\Omega\) of positive measure. As a result, 
\[ J_S(u)=\|c+Q_{S_0}u\|_{L^p(\Omega)}>c|\Omega|^{1/p}. \]
Now, let \(\{f_n\}\) be a normalized maximizing sequence and pass to a subsequence such that \(f_n\rightharpoonup f\). By the convergence formula proved before, we have
\[
\Lambda_{p,\Omega}(S)=\bigl\|c+Q_{S_0}f\bigr\|_{L^p(\Omega)}.
\]
If \(f=0\), this gives \(\Lambda_{p,\Omega}(S)=c|\Omega|^{1/p}\), contradicting the strict inequality above, therefore it must be \(f\neq0\). In particular, we claim that \(\|f\|_{L^2}=1\) --- in which case, the convergence formula gives \(J_S(f_n)\to J_S(f)\), hence the claimed attainment: \(J_S(f)=\Lambda_{p,\Omega}(S)\). Suppose otherwise \(0<\|f\|_{L^2}<1\) and renormalize to  \(g=f/\|f\|_{L^2}\). Since \(S_0\ge0\), we have
\[ Q_{S_0}g=\frac{1}{\|f\|_{L^2}^2}Q_{S_0}f\ge Q_{S_0}f. \]
If \(Q_{S_0}f\not\equiv0\) on \(\Omega\), then
\[ J_S(g)=\left\|c+\frac{1}{\|f\|_{L^2}^2}Q_{S_0}f\right\|_{L^p(\Omega)}>\|c+Q_{S_0}f\|_{L^p(\Omega)}=\Lambda_{p,\Omega}(S), \]
a contradiction. If \(Q_{S_0}f\equiv0\) on \(\Omega\), then \(\Lambda_{p,\Omega}(S)=c|\Omega|^{1/p}\), again contradicting the strict inequality above. 
\end{proof}

\begin{remark}
    The previous proof can be slightly generalized to show that, for every $K \in \cK$ and $c \in \C$ (without positivity assumptions), one has
    \[\Lambda_{p,\Omega}(c\,\mathrm{Id}+K) = \max_{0\le \rho\le 1,\ \|u\|_{L^2}=1} \left\|c+\rho Q_Ku\right\|_{L^p(\Omega)}.  \] 
    In particular, the maximum on the right-hand side is always attained, but existence of optimizers for the concentration problem associated with $S=c\,\mathrm{Id}+K$ requires a maximizer with $\rho =1$. The case $\rho<1$ leads to nonattainment of $\Lambda_{p,\Omega}(c\,\mathrm{Id}+K)$ due to loss of mass at infinity by a fraction of size $1- \rho$.
\end{remark}

Let us also introduce another family of noncompact operators for which the concentration problem has positive answer. The main motivation behind the introduction of this class revolves around the fact that, since we are concerned with local concentration problems, global compactness may be relaxed to a suitable notion of \textit{local} compactness to ensure existence of optimizers --- at least for positive operators: 
\begin{prop} \label{prop:PositiveLocalizedAverage}
Let \(S\in\cB\) be a positive operator, and fix a set \(\Omega\subset\rdd\) with finite positive measure, and $1 \le p < \infty$. 
If the mixed-state localization operator \(H_{\Omega,S}\) is compact, then the mapping $f \mapsto Q_Sf$ is sequentially weak-to-norm continuous $L^2(\rd)\to L^p(\Omega)$. 
\end{prop}

\begin{proof}
Let \(f_n\rightharpoonup f\) and set $r_n=f-f_n$, so that \(r_n\rightharpoonup0\). We need to prove that \(Q_Sf_n\to Q_Sf\) in \(L^p(\Omega)\). Since \(S\ge0\) and \(H_{\Omega,S} \in \cK\) by assumption, we have 
\[ \|Q_Sr_n\|_{L^1(\Omega)}= \int_\Omega\langle\alpha_z(S)r_n,r_n\rangle \dd{z} =\langle H_{\Omega,S}r_n,r_n\rangle\to0. \]
The uniform bound \(|Q_Sr_n(z)|\le\|S\|\bigl(\sup_n\|r_n\|_{L^2}^2\bigr)\) then upgrades the convergence to \(L^p(\Omega)\) for every finite \(p\). The desired conclusion follows after writing \(f_n=f+r_n\) and expanding $Q_Sf_n$, since the pure error term $Q_S r_n$ has just been handled and the mixed terms vanish by Lemma~\ref{lem-mixed}.
\end{proof}

\subsection{The $L^\infty$ optimization} \label{sec-linfty}
Let us consider now the endpoint case $p=\infty$, where we have an explicit characterization of the optimal concentration in terms of the numerical radius of $S$. 

\begin{prop}\label{Prop:NumericalRadius}
Let \(\Omega\subset\R^{2d}\) be such that $0<\abs{\Omega}$, and $S \in \cB$. Then
\[ \Lambda_{\infty,\Omega}(S)=w(S). \]
In particular, if the numerical radius of \(S\) is attained, then the supremum $\Lambda_{\infty,\Omega}(S)$ is attained as well. 
\end{prop}
\begin{proof}
To prove the upper bound \(\Lambda_{\infty,\Omega}(S)\le w(S)\), note that for every \(\|f\|_{L^2}=1\) and $z \in \rdd$ we have \(Q_Sf(z)=\langle \alpha_z(S)f,f\rangle\). Invariance of the numerical radius under unitary conjugation yields \( \abs{Q_Sf(z)}\le w(\alpha_z(S))=w(S)\), and thus the claimed inequality. 

For the reverse one, fix \(\varepsilon>0\) and choose a normalized vector \(g\in L^2(\R^d)\) such that
$|\langle Sg,g\rangle|>w(S)-\varepsilon$. Moreover, choose a Lebesgue density point \(z_0\) of \(\Omega\) and set \(f=\pi(z_0)g\). Then $Q_Sf(z_0)=\langle Sg,g\rangle$, and since \(Q_Sf\) is continuous on \(\R^{2d}\), there exists a neighborhood \(U\) of \(z_0\) such that $|Q_Sf(z)|>w(S)-2\varepsilon$ for all $z\in U$. By construction, the set \(U\cap\Omega\) has positive measure and thus 
\[
\|Q_Sf\|_{L^\infty(\Omega)}\ge w(S)-2\varepsilon.
\]
Letting \(\varepsilon\to 0\) proves the reverse inequality.

Finally, it is clear that if the numerical radius is attained at some normalized \(g\), the same construction with a density point \(z_0\) of \(\Omega\) gives a function \(f=\pi(z_0)g\) such that \(\|Q_Sf\|_{L^\infty(\Omega)}=w(S)\). 
\end{proof}

This result thus shows that existence of $L^\infty$ concentration optimizers for $Q_Sf$ is realized by all the \textit{numerical radius attaining} window operators $S$. In fact, the class $\NRA(X)$ of numerical radius attaining operators on a Banach space $X$ has been widely investigated in functional analysis, see for instance \cite{acosta,berg,capel,paya}. For our purposes, it is enough to emphasize that $\cK \subset \NRA$ on every Hilbert space, in particular $X=L^2(\rd)$ (which we omit for conciseness), but noncompact members are abundant as well --- including, for instance, scalar multiples of the identity and self-adjoint norm attaining operators. 

Let us also emphasize that the characterization in terms of $\NRA$ is almost optimal, in the sense that attainment of the \(L^\infty\)-concentration problem need not imply numerical-radius attainment for arbitrary finite-measure regions, while boundedness of $\Omega$ closes the gap. To be more concrete, in dimension \(d=1\) consider the multiplication operator \(S=M_\varphi\) associated with \(\varphi(x)\coloneqq 1-e^{-|x|}\). Then \(S\ge0\) and \( w(S)=\|S\|_{\cB}= \|\varphi\|_\infty=1\). On the other hand, we claim that \(S\notin\NRA\). Indeed, if \(\|g\|_{L^2}=1\) then \(\langle Sg,g \rangle <1\) and for any \(f\in L^2(\R)\) with \(\|f\|_{L^2}=1\) we have explicitly 
\[ 0 \le Q_Sf(x,\xi)=\int_{\mathbb R}\varphi(t-x)|f(t)|^2\dd{t} \le 1,\]
hence by dominated convergence we infer \( Q_Sf(x,\xi)\to1\) as \( x\to+\infty\). Choose then \(R_n\to+\infty\) such that \(Q_Sf(R_n,0)>1-\dfrac{1}{2n}\). By continuity of \(Q_Sf\) we can find \(\varepsilon_n>0\) sufficiently small to ensure
\[ (x,\xi)\in [R_n,R_n+\varepsilon_n]\times[0,\varepsilon_n] \implies Q_Sf(x,\xi)>1-\frac1n \] and 
\(\sum_n\varepsilon_n^2<\infty\). Consider now the set
\[ \Omega = \bigcup_{n=1}^\infty\bigl([R_n,R_n+\varepsilon_n]\times[0,\varepsilon_n]\bigr). \]
We have \(|\Omega|<\infty\) by construction, and for every \(n\) the set $[R_n,R_n+\varepsilon_n]\times[0,\varepsilon_n]\subset\Omega$ has positive measure and over there it holds \(Q_Sf>1-1/n\), therefore \( \|Q_Sf\|_{L^\infty} \ge1-\frac1n\) for every \(n\). Since \(Q_Sf\le1\), we conclude that \(\|Q_Sf\|_{L^\infty(\Omega)}=1=w(S)\), hence \(f\) attains the \(L^\infty\)-concentration supremum, although \(S\notin\NRA\). 

The feature of the previous example is that the essential supremum is approached along the unbounded part of \(\Omega\), rather than attained at a phase-space point. This option is ruled out in the case where \(\Omega\) is bounded: By covariance, continuity of \(z\mapsto Q_Sf(z)\), and compactness of \(\overline\Omega\) we see that attainment of the \(L^\infty\)-problem implies \(S\in\NRA\).

\section{Negative results}\label{sec:Negative}

The results in Section~\ref{Sec:direct} show that an optimizer of the concentration Problem~\eqref{Problem} exists for all compact operators \(S\), and for some compact perturbations of the identity. On the other hand, the recent results from~\cite{Stra-Svela-Trapasso-2025} show existence of an optimizer in the noncompact cases of the Wigner distribution (\(S=2^dP\)) and the Born--Jordan distribution (\(S=\mathcal{F}_W^{-1}(\mathrm{sinc}(\pi x\cdot \xi))\)). One could therefore expect that the concentration problem has an optimizer for any bounded \(S\). In this short section we prove that this is not the case. In particular, we exhibit some interesting counterexamples with different features. 

Firstly, we will show that existence of an optimizer can fail for (negative) compact perturbations of the identity. 
\begin{prop}\label{Prop:IdMinusGaussAllp}
Let \(1\le p<\infty\), let \(\Omega\subset\R^{2d}\) be measurable with \(0<|\Omega|<\infty\), and set \(
S=\mathrm{Id}-\varphi_0\otimes \varphi_0\), where \(\varphi_0(t)=2^{d/4}e^{-\pi |t|^2}\). Then the optimal value of Problem~\eqref{Problem} is \(|\Omega|^{1/p}\), but the supremum is not attained.
\end{prop}
\begin{proof}
By homogeneity it suffices to consider \(\|f\|_{L^2}=1\). In that case we have \(Q_Sf(z)=1-|V_{\varphi_0}f(z)|^2\), hence \(0\le Q_Sf(z)\le 1\) for every \(z\) and thus \(\|Q_Sf\|_{L^p(\Omega)}\le |\Omega|^{1/p}\). Let \(z_n\in \R^{2d}\) be a sequence such that \(|z_n|\to\infty\) and define \(f_n=\pi(z_n)\varphi_0\). Then \(\|f_n\|_{L^2}=1\), and by covariance of the spectrogram,
\[
|V_{\varphi_0}f_n(w)|^2=T_{z_n}\bigl(|V_{\varphi_0}\varphi_0|^2\bigr)(w).
\]
Note that \(|V_{\varphi_0}\varphi_0|^2\) is a centered Gaussian on phase space, therefore \(|V_{\varphi_0}f_n(w)|^2\to 0\) pointwise for almost every \(w\in \Omega\), and \(0\le |V_{\varphi_0}f_n|^2\le 1\). We thus infer \( Q_Sf_n(w)\to 1\) for a.e. \(w\in \Omega\), and by dominated convergence we conclude that the supremum equals \(|\Omega|^{1/p}\), since
\[
\|Q_Sf_n\|_{L^p(\Omega)}\to |\Omega|^{1/p}. 
\]
Assume by contradiction that a normalized maximizer \(f\) exists. Since \(0\le Q_Sf\le 1\) and \(\|Q_Sf\|_{L^p(\Omega)}=|\Omega|^{1/p}\), it should be \(Q_Sf=1\) almost everywhere on \(\Omega\). Equivalently, this means \( |V_{\varphi_0}f|^2=0 \) a.e. on \(\Omega\). Recall that the short-time Fourier transform of the Gaussian window \(\varphi_0\) coincides, up to a nonvanishing phase-space factor, with the Bargmann transform of \(f\), and is therefore real-analytic (indeed entire) in the corresponding complexified variables (see~\cite{Folland,Grochenig}). As such, vanishing on a set of positive measure implies \(V_{\varphi_0}f\equiv 0\), which in turn forces \(f=0\), contrary to the assumptions.
\end{proof}

The above counterexample clearly contrasts Proposition~\ref{prop:perturbation}. Similarly, we now show that the supremum in~\eqref{Problem} is not attained for the unitary operators \(S=\pi(z)\) whenever \(z\neq0\), in stark contrast to the special case \(S=\mathrm{Id}=\pi(0)\), and also to the almost unitary operator \(2^dP\). We start by computing the Cohen class representation associated with a time-frequency shift.
\begin{lemma}[{\cite[Remark 15]{Luef-Skrettingland-2021}}]
    Let \(z_0\in\R^{2d}\) and \(f\in L^2(\R^d)\). Then
    \begin{align*}
        Q_{\pi(z_0)}f(z)=e^{-2\pi i [z,z_0]}e^{\pi i x_0\cdot\xi_0} Af(-z_0),
    \end{align*}
    where \(A\) is the ambiguity function and \([z_0,z]\) denotes the symplectic inner product.
\end{lemma}
\begin{proof}
    It is well known, see for instance \cite[Lemma 3.1]{Bible1} (note that a different convention for the symplectic inner product is used there), that \(\alpha_z(\pi(z_0))=e^{2\pi i [z_0,z]}\pi(z_0)\) and \(\widecheck{\pi(z_0)}=\pi(-z_0)\). A straightforward computation using the definitions now shows that
    \begin{align*}
        Q_{\pi(z_0)}f(z)&=(f\otimes f)\star \widecheck{\pi(z_0)}=\tr\left(f\otimes f\ \alpha_z(\pi(z_0))\right)\\
        &=e^{2\pi i [z_0,z]}\tr\left(f\otimes f \ \pi(z_0)\right)=e^{2\pi i [z_0,z]}e^{\pi i x_0\cdot \xi_0}e^{-\pi i x_0\cdot \xi_0}\tr\left(f\otimes f \ \pi(z_0)\right)\\
        &=e^{-2\pi i [z,z_0]}e^{\pi i x_0\cdot \xi_0}\mathcal{F}_W\left(f\otimes f\right)(-z_0)\\
        &=e^{-2\pi i [z,z_0]}e^{\pi i x_0\cdot \xi_0}Af(-z_0). \qedhere
    \end{align*}
\end{proof}
The Cohen class related to a time-frequency shift has thus constant modulus, and is therefore quite simple. Nevertheless, the appearance of the ambiguity function obstructs the corresponding optimization problem.
\begin{prop}\label{AmbPoint}
    Let \(z_0=(x_0,\xi_0)\in \R^{2d}\setminus\{0\}\), and consider the Cohen class \(Q_{\pi(-z_0)}f(z)\). For $1 \le p < \infty$ and measurable $\Omega \subset \rdd$ with $0<|\Omega|<\infty$, the supremum
    \begin{align*}
        \sup_{f\in L^2(\R^d) \smo}\frac{\left(\int_{\Omega}|Q_{\pi(-z_0)}f(z)|^p \dd{z}\right)^{1/p} }{\|f\|_{L^2}^2}=\sup_{f\in L^2(\R^d) \smo}\frac{|\Omega|^{1/p} |Af(z_0)|}{\|f\|_{L^2}^2}
    \end{align*}
    is \(|\Omega|^{1/p} \), but the supremum is not attained.
\end{prop}
\begin{proof}
    From Cauchy-Schwarz we have \(\frac{|\Omega|^{1/p} |Af(z_0)|}{\|f\|_{L^2}^2}\leq \frac{|\Omega|^{1/p} \|f\|_{L^2}^2}{\|f\|_{L^2}^2}=|\Omega|^{1/p} ,\) but we also have to show that this bound is in fact the least upper bound. Let us for a moment consider the special case \(z_0=(x_0,0)\), where \(|Q_{\pi(-x_0,0)}f(z)|=|Af(x_0,0)|\). If we define the dilated Gaussians \begin{align*}
        f_{\lambda}(y)=\left(\frac{\sqrt{2}}{\lambda}\right)^{\frac{d}{2}}e^{-\pi \frac{|y|^2}{\lambda^2}},
    \end{align*} then it is well known that \begin{align*}
        |Af_{\lambda}(x,\xi)|=e^{-\frac{\pi|x|^2}{2\lambda^2}}e^{-\frac{\pi\lambda^2|\xi|^2}{2}}.
    \end{align*}
    At the point \((x_0,0)\) we thus have \begin{align*}
      |Af_{\lambda}(x_0,0)|=e^{-\frac{\pi|x_0|^2}{2\lambda^2}}.  
    \end{align*}
    For the general case, let \(U\in \mathrm{Sp}(2d,\R)\cap O(2d) \simeq U(d)\) be the symplectic rotation such that \(U(|z_0|,0)=z_0=(x_0,\xi_0)\), and let \(\mu(U)\) be the corresponding metaplectic operator (cf.\ ~\cite[Eq. 4.23]{Folland}). By the symplectic covariance of the ambiguity function (see~\cite[Proposition 4.28]{Folland} or \cite[Corollary 13.1.2.4]{DG_book_BJ}) we have \begin{align*}
        |Q_{\pi(-z_0)}(\mu(U)f)(z)|=|A(\mu(U)f)(z_0)|=|Af(U^{-1}z_0)|=|Af(|z_0|,0)|. 
    \end{align*}
    So for any \(z_0\) we have\begin{align*}
        |A(\mu(U)f_{\lambda})(z_0)|=|Af_{\lambda}(|z_0|,0)|=e^{-\frac{\pi|z_0|^2}{2\lambda^2}}.  
    \end{align*}
    As such, \begin{align*}
        \sup_{f\in L^2(\R^d)}\frac{|\Omega|^{1/p} |Af(z_0)|}{\|f\|_{L^2}^2}\geq\frac{|\Omega|^{1/p} }{\|\mu(U)f_\lambda\|_{L^2}^2} e^{-\frac{\pi|z_0|^2}{2\lambda^2}}=|\Omega|^{1/p} e^{-\frac{\pi|z_0|^2}{2\lambda^2}}
    \end{align*}
    for all \(\lambda>0\).  Letting \(\lambda\rightarrow\infty\) shows that the supremum must be \(|\Omega|^{1/p} \).  However, the radar correlation estimate states that if \(z_0\neq 0\) then \begin{align*}
        |Af(z)|<|Af(0)|=\|f\|_{L^2}^2,
    \end{align*}
    and thus there is no nonzero function such that \(|Af(z_0)|=\|f\|_{L^2}^2\). Consequently, there is no optimizer for \(Q_{\pi(-z_0)}f\) when \(z_0\neq 0\).
\end{proof}

\begin{remark} In the endpoint case $p=\infty$, the same obstruction gives \[\Lambda_{\infty,\Omega}(\pi(-z_0))=1 \qquad \forall\ z_0\ne0,\] and the supremum is not attained. Indeed, the upper bound is \(|Af(z_0)|/\|f\|_{L^2}^2\le1\), while the lower bound is obtained by resorting to the same dilated Gaussian sequence. Equality would force \(\pi(z_0)f\) to be a unimodular multiple of \(f\), which is impossible for a nonzero \(L^2\)-function when \(z_0\ne0\).
\end{remark}

\section{Structural properties behind existence of optimizers} \label{sec-structure}
In light of the two classes of counterexamples just examined, as well as the results from the previous section, we can draw some general conclusions about the present state of the concentration problem~\eqref{Problem}. 

Firstly, it is not clear which properties of \(S\) imply the existence of an optimizer for \(Q_Sf\): The compact operators, the identity, \(2^dP\) and the Born--Jordan operator do not appear to share any properties that differentiate them from the time-frequency shifts and the arbitrary compact perturbations. Moreover, the operators for which \(Q_Sf\) has an optimizer seem to do so for different reasons or by means of different arguments --- some associate with continuous concentration functionals, while some do not. At the moment, it seems therefore inevitable that in order to fully understand the concentration problem for Cohen's class distributions with noncompact \(S\) one must treat each operator \(S\) on a case-by-case basis. 

While a full classification seems currently out of reach (if not impossible at all), in this section we make some progress in connection with this program. 

\subsection{General facts} To be more definite, consider the following classes of operators for fixed $1 \le p < \infty$ and \(0<|\Omega|<\infty\): 
\[\cBsup \coloneqq \Bigg\{S\in \cB : \Lambda_{p,\Omega}(S) = \sup_{f\in L^2(\R^d)\setminus\{0\}}\frac{\norm{Q_S f}_{L^p(\Omega)}}{\|f\|_{L^2}^2} \text{ is finite and attained}\Bigg\},\] 
\[\cBwc \coloneqq \{S\in \cB :  f_n \rightharpoonup f \implies \|Q_S f_n\|_{L^p(\Omega)} \to \|Q_Sf\|_{L^p(\Omega)}\},\]
\[\cBwtn\coloneqq\{S\in\cB : f_n \rightharpoonup f \implies \norm{Q_S f_n - Q_Sf}_{L^p(\Omega)} \to 0\},\]
\[\cBpt \coloneqq \{S\in \cB :  f_n\rightharpoonup f \implies Q_Sf_n(z)\to Q_Sf(z)\text{ for every }z\in \R^{2d}\}. \]
We clearly have the following chain of inclusions, the first being obtained in the proof of Theorem~\ref{CompactExistence} and the last by the direct method of calculus of variations:
\[ \cK \subseteq \cBpt \subseteq \cBwtn \subseteq \cBwc \subseteq \cBsup. \]
The previous findings naturally suggest some problems about the structure of these sets, which are explored below. Before, let us make some preliminary remarks of general nature. We note that $\cBsup$ is a balanced set and $\C\ \mathrm{Id} \subset \cBsup$, since for every $c \in \C$ and $f \in L^2(\rd)\smo$ one has
\[\frac{\|Q_{c\,\mathrm{Id}}f\|_{L^p(\Omega)}}{\|f\|_{L^2}^2}=|c|\,|\Omega|^{1/p}. \]
On the other hand, $c \mathrm{Id} \in \cBwc$ if and only if $c=0$, since for any weakly null normalized sequence $f_n \rightharpoonup 0$ we have $Q_{c \mathrm{Id}}f_n(z)=c$ and $Q_{c\mathrm{Id}}0(z)=0$. As a consequence, $\cBwc \subsetneq \cBsup$. 

These facts also imply that convexity of $\cBsup$ fails due to Proposition~\ref{Prop:IdMinusGaussAllp}, and that the naive conjecture about \(\mathcal{K}\) being the largest norm-closed linear subspace contained in \(\cBsup\) has a negative answer. The same conjecture with \(\cBsup\) replaced by \(\cBwc\) fails as well, in view of the Born--Jordan results in \cite{Stra-Svela-Trapasso-2025,Stra-Svela-Trapasso-2026}, since \(S_{\mathrm{BJ}}=\mathcal{F}_W^{-1}(\mathrm{sinc}(\pi x\cdot \xi)) \in \cBwc \setminus \cK\) in subcritical regimes --- and thus the whole ray $\C S_{\mathrm{BJ}}$ belongs to $\cBwc$. 

On the positive side, we have some interesting characterizations. Let us first introduce the notion of essential concentration value: 
\begin{definition}
For \(S\in\cB\), \(1\le p<\infty\), and \(\Omega \subset \rdd\) with \(0<|\Omega|<\infty\), define
\[ \Lambda^{\mathrm{ess}}_{p,\Omega}(S) \coloneqq \sup\left\{ \limsup_{n\to\infty}\|Q_Sf_n\|_{L^p(\Omega)}: \|f_n\|_{L^2}=1,\ f_n\rightharpoonup0 \right\}. \]
\end{definition}
Heuristically, this provides a quantitative measure of the fraction of the concentration that survives in $\Omega$ if all maximizing mass is allowed to escape weakly to infinity. For instance, it is easy to show that for every $K \in \cK$ and $c \in \C$ one has
\[ \Lambda^{\mathrm{ess}}_{p,\Omega}(c\,\mathrm{Id}+K)=|c|\,|\Omega|^{1/p}. \]
It is then natural to wonder whether this quantity is involved in the different behaviors illustrated in Propositions \ref{prop:perturbation} and \ref{Prop:IdMinusGaussAllp}. This is precisely the case, as a consequence of the following couple of results. Let us start with a simple gap criterion for existence of optimizers. 

\begin{prop}\label{Prop:StrictGap}
With the notation introduced above,
\[ \Lambda_{p,\Omega}(S)>\Lambda^{\mathrm{ess}}_{p,\Omega}(S) \implies S \in \cBsup. \]
\end{prop}

\begin{proof}
    Let \((f_n)\) be a normalized maximizing sequence and pass to a subsequence such that \(f_n\rightharpoonup f\), and set \(c=\|f\|_{L^2}\). If \(c=1\), then \(f_n\to f\) strongly in \(L^2(\rd)\), and strong continuity of the map \(f\mapsto Q_Sf\) from \(L^2(\rd)\) to \(L^p(\Omega)\) implies \( \|Q_Sf\|_{L^p(\Omega)}=\Lambda_{p,\Omega}(S)\), so \(f\) is an optimizer.

It remains to rule out the case \(c<1\). Setting \(r_n=f_n-f\), we thus have \(r_n\rightharpoonup0\) and \( \|r_n\|_{L^2}^2\to 1-\|f\|_{L^2}^2\). The triangle inequality gives 
\[  \Lambda_{p,\Omega}(S)  \le  \|Q_Sf\|_{L^p(\Omega)}+\limsup_{n\to\infty}\|Q_Sr_n\|_{L^p(\Omega)},\]
since the mixed terms involving $f$ and $r_n$ vanish in \(L^p(\Omega)\) by Lemma \ref{lem-mixed}, hence implying 
\[ \limsup_{n\to\infty}\|Q_Sr_n\|_{L^p(\Omega)} \ge (1-c^2)\Lambda_{p,\Omega}(S). \]
Set \(d_n=\|r_n\|_{L^2}\), so that \(d_n\to d \coloneqq(1-c^2)^{1/2}>0\). Define accordingly \(g_n=r_n/d_n\), hence \(g_n\rightharpoonup0\), \(\|g_n\|_{L^2}=1\), and
\[ \limsup_{n\to\infty}\|Q_Sr_n\|_{L^p(\Omega)} =d^2\limsup_{n\to\infty}\|Q_Sg_n\|_{L^p(\Omega)}
\le (1-c^2)\Lambda^{\mathrm{ess}}_{p,\Omega}(S), \]
or equivalently 
\[ \Lambda_{p,\Omega}(S) \le c ^2\Lambda_{p,\Omega}(S)+(1-c^2)\Lambda^{\mathrm{ess}}_{p,\Omega}(S), \]
which contradicts the strict gap assumption. Thus we must have \(c=1\) and so \(S\in \cBsup\).
\end{proof}

The essential concentration value turns out to be useful also in the following characterization. 

\begin{prop} \label{prop:wtn} The following conditions are equivalent for $S \in \cB$:
\begin{itemize}
    \item $S \in \cBwtn$.
    \item $S \in \cBwc$.
    \item $\Lambda^{\mathrm{ess}}_{p,\Omega}(S)=0$.
\end{itemize} Moreover, $\cBwtn= \cBwc= \ker \Lambda^{\mathrm{ess}}_{p,\Omega}$ is a norm-closed complex linear subspace of $\cB$. 
\end{prop} 

\begin{proof} 
The non-trivial inclusion to be proved is $\cBwc \subset \cBwtn$. To this aim, \(S\in\cBwc\) and let \(f_n\rightharpoonup f\). After setting $r_n\coloneqq f_n-f$, we have \(r_n\rightharpoonup0\) and we can write 
\[ Q_Sf_n-Q_Sf=Q_Sr_n+Q_S(f,r_n)+Q_S(r_n,f). \]
Since \(S\in\cBwc\) we have $\|Q_Sr_n\|_{L^p(\Omega)}\to\|Q_S0\|_{L^p(\Omega)}=0$, and the mixed terms vanish in $L^p(\Omega)$ by Lemma \ref{lem-mixed} as well. 

The same argument shows that if $\Lambda^{\mathrm{ess}}_{p,\Omega}(S)=0$ then \(Q_Sr_n\to0\) in \(L^p(\Omega)\) (the mixed terms vanish), hence \(S\in\cBwtn\). In fact, it suffices to prove that every subsequence of \((r_n)\) has a further subsequence along which this convergence holds. Let then \((r_{n_k})\) be any subsequence, and up to further subsequences we may assume $\|r_{n_k}\|_{L^2}\to c$
for some \(c\ge0\). If \(c=0\) the claim is obvious since $S \in \cB$, so we assume \(c>0\). Therefore,
\[ g_k\coloneqq \frac{r_{n_k}}{\|r_{n_k}\|_{L^2}} \implies g_k\rightharpoonup0, \quad
\|g_k\|_{L^2}=1. \]
To conclude, the assumption $\Lambda^{\mathrm{ess}}_{p,\Omega}(S)=0$ yields
\[ \|Q_Sr_{n_k}\|_{L^p(\Omega)} = \|r_{n_k}\|_{L^2}^2\|Q_Sg_k\|_{L^p(\Omega)} \to0, \]
hence \(S\in\cBwtn\). The converse is trivial: If \(S\in\cBwc\) then every normalized weakly null sequence \((f_n)\) satisfies \(\|Q_Sf_n\|_{L^p(\Omega)}\to0\), so $\Lambda^{\mathrm{ess}}_{p,\Omega}(S)=0$. 

\(\cBwtn\) is clearly a complex linear subspace of $\cB$. For norm closedness, consider a sequence \(S_m\in\cBwtn\) such that \(S_m\to S\) in operator norm. The goal is to prove that \(S\in\cBwtn\), namely that \(f_n\rightharpoonup f\) implies $\|Q_Sf_n - Q_Sf\|_{L^p(\Omega)} \to 0$. By the triangle inequality we have
\[ \|Q_Sf_n - Q_Sf\|_{L^p(\Omega)} \le \|Q_{S-S_m}f_n-Q_{S-S_m}f\|_{L^p(\Omega)}+\|Q_{S_m}f_n-Q_{S_m}f\|_{L^p(\Omega)}, \]
and it suffices to take $m$ and $n$ large enough to ensure that, for arbitrarily chosen $\varepsilon>0$, 
\[ \|Q_{S-S_m}f_n-Q_{S-S_m}f\|_{L^p(\Omega)} \le |\Omega|^{1/p}\|S-S_m\|_{\cB}\bigl(\sup_n\|f_n\|_{L^2}^2+\|f\|_{L^2}^2\bigr) < \varepsilon/2, \] and $\|Q_{S_m}f_n-Q_{S_m}f\|_{L^p(\Omega)} < \varepsilon/2$. 
\end{proof}

\begin{remark} The previous proof actually gives something more. Since failure of compactness of maximizing sequences is measured by \(\Lambda^{\mathrm{ess}}_{p,\Omega}\), and the weakly continuous class is exactly the kernel of this escape level, existence of optimizers is governed by whether a maximizing sequence can retain a nonzero weak profile. As such, if \(S\notin\cBsup\) then no normalized maximizing sequence can have a weakly convergent subsequence whose weak limit has norm \(1\), and the strict-gap criterion forces $\Lambda_{p,\Omega}(S)=\Lambda^{\mathrm{ess}}_{p,\Omega}(S)$. 
\end{remark}

In light of the discussion, it is natural to introduce the following ``strict-gap'' class: 
\[ \cG \coloneqq \{S\in\cB : \Lambda_{p,\Omega}(S)>\Lambda^{\mathrm{ess}}_{p,\Omega}(S)\}. \]
We have seen that \(\cG\subset\cBsup\), and it is not difficult to show that \(\cG\) is open in the operator norm, due to continuity of $\Lambda^{\mathrm{ess}}_{p,\Omega}$ --- which is in turn a consequence of the Lipschitz estimate
\[ |\Lambda_{p,\Omega}(S)-\Lambda_{p,\Omega}(T)|\le |\Omega|^{1/p}\|S-T\|_{\cB}, \qquad S,T \in \cB,\]
applied uniformly to normalized weakly null sequences. Moreover, every nonzero \(S\in\cBwc\) is an interior point of \(\cBsup\), since \(\Lambda^{\mathrm{ess}}_{p,\Omega}(S)=0\) by the preceding proposition, while \(\Lambda_{p,\Omega}(S)>0\) by Lemma~\ref{lem:nonzero-window-nonvanishing}. While attainment is governed by strict gap, the threshold set
\[ \mathcal T (p,\Omega) = \{S:\Lambda_{p,\Omega}(S)=\Lambda^{\mathrm{ess}}_{p,\Omega}(S)\} \]
contains both attained and nonattained examples. The identity operator is attained at threshold, whereas \(\mathrm{Id}-\varphi_0\otimes\varphi_0\) and the nontrivial Weyl shifts exhibit escape and nonattainment. Thus equality with the essential level is the only place where nonattainment can occur, but it is not by itself a nonattainment criterion. To be more precise, we have 
\[ \cBsup= \cG \sqcup\cTatt, \qquad \cB\setminus\cBsup=\cTloss, \] where we set
\[ \cTatt \coloneqq \{S \in \cB:\Lambda_{p,\Omega}(S)=\Lambda^{\mathrm{ess}}_{p,\Omega}(S)\quad \text{ and } \quad  \Lambda_{p,\Omega}(S)\text{ is attained}\}, \]
\[ \cTloss \coloneqq \{S \in \cB:\Lambda_{p,\Omega}(S)=\Lambda^{\mathrm{ess}}_{p,\Omega}(S) \quad \text{ and } \quad \Lambda_{p,\Omega}(S)\text{ is not attained}\}. \]
These findings complement the previous chain:  
\[ \cK = \cBpt\subsetneq\cBwtn=\cBwc=\ker\Lambda^{\mathrm{ess}}_{p,\Omega} \subsetneq\cBsup. \]
The essential optimal value $\Lambda^{\mathrm{ess}}_{p,\Omega}$ then gives a norm on the quotient space $\cB/\cBwc$, as well as a seminorm on the Calkin algebra $\cB/\cK$ with kernel $\cBwc/\cK$ (recall that $\Lambda^{\mathrm{ess}}_{p,\Omega}$ is continuous).  

It is important to stress the dependence on $p$, $\Omega$ and $d$ of these results. In this connection, the results in \cite{Stra-Svela-Trapasso-2025} about the Wigner window $S_{\mathrm{W}} = 2^dP$ can be phrased as follows: 
\[ S_{\mathrm{W}}\in \cBsup, \qquad S_{\mathrm{W}}\notin\cK,\qquad S_{\mathrm{W}}\notin\cBwc \]
but it is open whether $S_{\mathrm{W}} \in \cG$ or $S_{\mathrm{W}} \in \cTatt$ in general. As already commented in \cite{Stra-Svela-Trapasso-2025}, there is reason to believe that $\cG$ is the correct regime, although $\cTatt$ cannot be excluded a priori. An advantage of the QHA formulation is that the strict-gap criterion allows us to make this investigation more quantitative, hence we take this occasion to support this claim with evidence for sufficiently large or small balls in Appendix \ref{app-wig}.

In the same paper we also have $S_{\mathrm{BJ}}^{(1)} \in \cBwtn$ for the one-dimensional Born--Jordan window. On the other hand, the results in \cite{Stra-Svela-Trapasso-2026} about the higher-dimensional Born--Jordan window \(S_{\mathrm{BJ}}^{(d)}\) fall beyond the bounded operator setting, since for \(S\in\cB\) and \(1\le p<\infty\) one always has \(\Lambda_{p,\Omega}(S)< \infty\). On the other hand, if one
enlarges the class of admissible windows to distributional Cohen-class kernels, then a fourth alternative appears:
\[  \cBubd \coloneqq \{S \in \mathfrak{S}':\Lambda_{p,\Omega}(S)=+\infty\}.\]
Therefore, the classification depends on the critical exponent 
\[p_*(d)=\begin{cases}
    \infty & (d=1,2) \\ \dfrac{2d}{d-2} & (d\ge 3),  \end{cases}\]
    and reads as follows: 
\[ S_{\mathrm{BJ}}^{(d)} \begin{cases} \in \cBwtn & \text{ for all } d \text{ and } 1\le p<p_*(d) \\  \in \cBubd & \text{ for } d\ge 3 \text{ and } p>p_*(d). \end{cases}\]
Let us conclude this discussion with some comments concerning the endpoint case $p=\infty$, not covered in the previous finite-$p$ analysis. The results from Section \ref{sec-linfty} can be summarized here by introducing 
\[ \cB_{\mathrm{sup}}(\infty,\Omega) = \left\{S \in \cB:\text{ there is }f\in L^2(\mathbb R^d) \text{ with } \|f\|_{L^2}=1 \text{ such that } \|Q_Sf\|_{L^\infty(\Omega)}=w(S)\right\}. \]
We have seen that $\cB_{\mathrm{sup}}(\infty,\Omega) = \NRA$ in the case where $\Omega$ is bounded. In general, for unbounded finite-measure $\Omega$ the endpoint class can strictly contain $\NRA$, and the classification for $p=\infty$ is
\[ \cB=\NRA\ \sqcup\ \bigl(\cB_{\mathrm{sup}}(\infty,\Omega)\setminus\NRA\bigr)\ \sqcup\ \bigl(\cB\setminus\cB_{\mathrm{sup}}(\infty,\Omega)\bigr).
\]
If one enlarges the admissible windows to distributional Cohen-class kernels, one may additionally define
\[\cB_{\mathrm{ubd}}(\infty,\Omega)\coloneqq\{S\in\mathfrak S':\Lambda_{\infty,\Omega}(S)=+\infty\}.\]
In particular, the critical Born--Jordan endpoint in dimension $d=2$ satisfies $S_{\mathrm{BJ}}^{(2)}\in \cB_{\mathrm{ubd}}(\infty,\Omega)$.

\subsection{The compactness barrier}
As already discussed, the Born--Jordan yields strictness of the inclusion $\cK \subsetneq \cBwc$, but leaves open the gap between $\cBpt$ and $\cBwc$. This problem is actually solved in light of the following noteworthy characterization. 

\begin{prop}
    $\cBpt = \cK$. 
\end{prop}
\begin{proof}
In view of Theorem \ref{CompactExistence}, we need to prove $\cBpt \subset \cK$ only. It is well known that \(S\in\cK\) if and only if \(f_n\rightharpoonup0 \implies \|Sf_n\|\to0\). By polarization, this is equivalent to \(\langle Sf_n,f_n\rangle\to0\) for every weakly null sequence $f_n$. Now, if $S \in \cBpt$ and \(f_n\rightharpoonup0\), by the definition of \(\cBpt\) applied at \(z=0\) we have
\[\langle Sf_n,f_n\rangle=Q_Sf_n(0)\to Q_S0(0)=0,\]
that is the claim. 

\end{proof}

From this result we deduce that noncompact examples with weakly continuous concentration functional cannot be detected pointwise, hence they must be genuinely $L^p(\Omega)$ based. Before further exploring this direction, let us point out that the compactness threshold in $\cBwc$ is sharp at least when one asks some additional regularity. Recall from \cite{Luef-Skrettingland-2021} the class of uniformly continuous operators with respect to shift action: 
\[ \Cu \coloneqq \{S \in \cB : \rdd \ni z \mapsto \alpha_z(S) \in \cB \text{ is continuous}\}. \]

\begin{prop}
    With the notation introduced above, we have $\cBwc \cap \Cu= \cK$. 
\end{prop}

\begin{proof}
We first note that \( \cK\subseteq \cBwc\cap \Cu \). Indeed, $\cK \subset \cBwc$ by Theorem~\ref{CompactExistence}, and finite-rank operators belong to \(\Cu\) since \(z\mapsto \pi(z)u\)
is norm-continuous for every \(u\in L^2(\rd)\). By density, we have $\cK \subset \Cu$ as well. The non-trivial inclusion is thus the opposite one, and we prove now that if \( S\in\cBwc\cap\Cu\) then \(S\in\cK\). 

\bigskip

\noindent \textbf{Step 1.} \emph{Vanishing at infinity of $Q_S\varphi_0$.} 
Let \(\varphi_0\) be the usual normalized Gaussian, then set \( A_0 \coloneqq \varphi_0\otimes\varphi_0\) and \[ F(z) \coloneqq Q_S\varphi_0(z) = (A_0\star\widecheck S)(z) = \langle \alpha_z(S)\varphi_0,\varphi_0\rangle. \]
Since \(S\in\Cu\) we have that \(F\) is a uniformly continuous function on \(\rdd\). We claim that \(F\) vanishes at infinity, that is \(F\in C_0(\rdd)\). Indeed, given any sequence \((z_n)\subset\rdd\) satisfying \(|z_n|\to\infty\), we have \( g_n \coloneqq \pi(z_n)\varphi_0\rightharpoonup0 \) in $L^2(\rd)$. Since \(S\in\cBwc\), weak continuity 
gives \( \|Q_Sg_n\|_{L^p(\Omega)}\to0 \), and by covariance of \(Q_S\) we also have
\[ |Q_Sg_n(z)| = |Q_S\varphi_0(z-z_n)| = |F(z-z_n)|. \]
As a result, we infer \( \|F\|_{L^p(\Omega-z_n)}\to0\) as \(|z_n|\to\infty\). 

Suppose, for the sake of contradiction, that \(F\notin C_0(\rdd)\). Since \(F\) is uniformly continuous, there exist \(\eta>0\), a sequence \(w_n\to\infty\), and a radius \(r>0\) such that \( |F(z)|\ge \eta/2\) for every \(z\in B(w_n,r)\) and \(n\). If \(\zeta_0\) is a density point of \(\Omega\), setting \( z_n \coloneqq \zeta_0-w_n\) yields \(|z_n|\to\infty\) and 
\[ \|F\|_{L^p(\Omega-z_n)}^p \ge \int_{(\Omega-z_n)\cap B(w_n,r)} |F(z)|^p \dd{z} \ge (\eta/2)^p |\Omega\cap B(\zeta_0,r)|>0, \] where we used that \(|(\Omega-z_n)\cap B(w_n,r)| = |\Omega\cap B(\zeta_0,r)|>0\). This bound contradicts \(\|F\|_{L^p(\Omega-z_n)}\to0\), proving that \(F=Q_S\varphi_0\in C_0(\rdd)\) as claimed. 

\bigskip 

\noindent \textbf{Step 2.} \emph{Pitt compactness.}
Let us note that, since \(A_0=\widecheck A_0\) and
\[
A_0 \star \widecheck S (z)=\tr (A_0\alpha_z(S)) = \tr(S\alpha_{-z}(A_0)) = S \star A_0(-z), \qquad z \in \rdd,
\]
the following equivalence holds:
\[
A_0\star\widecheck S\in C_0(\rdd) \iff S\star A_0\in C_0(\rdd).
\]
We can finally invoke the QHA Pitt compactness criterion, cf.~\cite[Theorem 5.2]{Luef-Skrettingland-2021}: \(R\in\cK\) if and only if \(R\in\Cu\) and \( R\star T\in C_0(\rdd)\) for some trace-class operator \(T \in \cS^1\) with $\cF_W(T)(z) \ne 0$ for all $z \in \rdd$. Applying this with \(R=S\) and \(T=A_0\), since $\cF_W(A_0)(z)=e^{-\pi |z|^2/2}$, gives the desired conclusion \(S\in\cK\).
\end{proof}

\subsection{Beyond the compactness barrier}
The last two propositions isolate a true compactness barrier from two complementary angles: Any noncompact weakly continuous window must be singular from the QHA point of view, namely invisible to the pointwise topology and lie outside the uniformly continuous operator class. 

In this connection, it is worth highlighting that the Born--Jordan ray is an explicit noncompact subset of $\cBwc \setminus \Cu$. As such, the weak continuity mechanism in this case is inherently different from the one encoded by the regularity \(\Cu\), which ultimately accounts for compactness --- hence pointwise convergence.

Some questions naturally arise at this point. First, what is the structural nature behind the weak continuity of the Born--Jordan concentration functional? Is the latter shared by other operators? More importantly, are there different mechanisms? In this section we explore these problems. 

\subsubsection{The squeeze representation of the Born--Jordan window.} 
Assume \(d=1\) and set \(D_sf(x)\coloneqq e^{s/2}f(e^sx)\) for \(s \in \R\). After setting  
\[\tau=\tau(s)\coloneqq \frac{e^s}{1+e^s},\] a straightforward computation shows that the \(\tau\)-Wigner window can be written as \(S_\tau=2\cosh(s/2)D_sP\), since 
\[\langle D_sPf,f\rangle=e^{s/2}\int_{\mathbb R}f(-e^sx)\overline{f(x)} \dd{x}=\sqrt{\tau(1-\tau)}\int_{\mathbb R}f(\tau t)\overline{f(-(1-\tau)t)} \dd{t}.\] 
Averaging over \(\tau\in(0,1)\) then gives a representation of the Born--Jordan window: 
\[S_{\mathrm{BJ}}=\int_0^1S_\tau \dd{\tau}=\int_{\mathbb R}\frac{1}{2\cosh(s/2)}D_sP \dd{s}. \] 
An interesting perspective is offered here in terms of functional calculus: If \(K \) is the self-adjoint generator of the squeeze group $D_s=e^{-isK}$, since the latter commutes with \(P\) we can write
\[ S_{\mathrm{BJ}}=m_{\mathrm{BJ}}(K)P,\qquad m_{\mathrm{BJ}}(\lambda)\coloneqq\int_{\mathbb R}\frac{1}{2\cosh(s/2)}e^{-is\lambda} \dd{s} = \frac{\pi}{\cosh(\pi\lambda)}. \] 
The previous computation, although elementary, sheds some light behind the mechanism leading to weak continuity of the concentration functional. Indeed, even if \(2P \notin \cBwc\) (see \cite{Stra-Svela-Trapasso-2025}), in the Born--Jordan window this factor is smoothed by the multiplier \(m_{\mathrm{BJ}}(K)\) resulting from averaging over the squeeze group. This accounts for cancellation effects that are enough to make $S_{\mathrm{BJ}}$ fall into $\cBwc$ but not to produce a compact operator, as illustrated by the following general result --- whose assumptions are satisfied by the multiplier $m_{\mathrm{BJ}}$. 
\begin{lemma}\label{lem:noncomp}
Let \(K\) be the generator of the dilation group $D_s$ on \(L^2(\rd)\). For every unitary operator $U$ and \(m\in C(\mathbb R)\cap L^\infty(\mathbb R)\) not identically zero, the averaged multiplier \(m(K)U\) is not compact.
\end{lemma}
\begin{proof}
The crucial mechanism here is that, by the spectral theorem, since \(\sigma(K)=\R\) then the spectral projection \(E_K(I)\) has infinite rank for every nonempty open interval \(I\subset\R\). The assumptions on $m$ then imply that there exist \(\lambda_0\in\R\) and \(c_0>0\) such that \(|m(\lambda_0)|>2c_0\), and also a nonempty open interval \(I\subset\R\) such that \(\lambda_0 \in I\) and \(|m(\lambda)|\ge c_0\) for every $\lambda\in I$. Since \(E_K(I)\) has infinite rank, there is an orthonormal sequence \((e_n)_{n=1}^\infty\subset \operatorname{Ran}E_K(I)\) such that for \(n\ne k\) the vector \(e_n-e_k\) also belongs to \(\operatorname{Ran}E_K(I)\), and therefore its spectral measure is supported in \(I\). We can now resort to the spectral theorem to obtain
\[\|m(K)(e_n-e_k)\|_{L^2}^2=\int_I |m(\lambda)|^2 \dd{\mu_{e_n-e_k}(\lambda)}\ge c_0^2\|e_n-e_k\|_{L^2}^2=2c_0^2.\]
As a consequence we have that \((m(K)e_n)_{n=1}^\infty\) has no Cauchy subsequence, but since \((e_n)\) is bounded the conclusion is impossible if \(m(K) \in \cK\), therefore \(m(K)\) is not compact as claimed. To conclude, if \(m(K)U\) were compact then \[(m(K)U)U^*=m(K)\] would be compact as well, a contradiction. 
\end{proof}

More generally, for any $h \in L^1(\R)\smo$ one is led to consider 
\[S_h\coloneqq\int_{\mathbb R}h(s)D_sP \dd{s}=m_h(K)P,\qquad m_h(\lambda)\coloneqq\int_{\mathbb R}h(s)e^{-is\lambda} \dd{s},\] as a natural family of noncompact windows for which weak continuity of the concentration functional may still hold under suitable assumptions on the multiplier \(m_h\), or equivalently on the averaging kernel \(h\). To determine such conditions, it is wise to revisit the argument in \cite{Stra-Svela-Trapasso-2025} for the Born--Jordan kernel. A closer inspection actually reveals that the same proof extends, after suitable adaptations, to more general averages of the $\tau$-Wigner distributions: 
\begin{prop}
  Let $\mu$ be a finite complex Borel measure on $(0,1)$ such that $\abs{\mu}$ has no atoms and 
  
  \[ M_\mu\coloneqq \int_0^1\frac{1}{\sqrt{\tau(1-\tau)}} \dd{|\mu|(\tau)}<\infty. \]
For $f,g \in L^2(\R)$ define the corresponding (Bochner) average
\[ W_\mu(f,g)\coloneqq\int_0^1W_\tau(f,g) \dd{\mu(\tau)}, \qquad W_\mu f\coloneqq W_\mu(f,f). \]
Then, for every measurable $\Omega\subset\mathbb R^2$ with $0<|\Omega|<\infty$ and every $1\le p<\infty$, the correspondence $f \mapsto W_\mu f$ is sequentially weak-to-norm continuous  $L^2(\R) \to L^p(\Omega)$.
\end{prop}
To connect this statement with operator windows, define
\[ S_\mu\coloneqq\int_0^1S_\tau \dd{\mu(\tau)} \]
as a weak operator integral. Since \(W_\tau(f,g)=Q_{S_\tau}(f,g)\), by linearity we infer 
\[ W_\mu(f,g)=\int_0^1W_\tau(f,g) \dd{\mu(\tau)}=Q_{S_\mu}(f,g), \]
hence the previous result is precisely giving \(S_\mu\in\cBwc\). As an immediate corollary, we have that \textit{every} kernel $h \in L^1(\R)\smo$ produces a noncompact operator $S_h \in \cBwc$. Indeed, recalling that \(\tau=\dfrac{e^s}{1+e^s}\) and \(S_\tau=2\cosh(s/2)D_sP\), we have
\[ S_h=\int_{\R} h(s)D_sP \dd{s} = \int_0^1a_h(\tau)S_\tau \dd{\tau}, \qquad a_h(\tau)=\frac{h(\log(\tau/(1-\tau)))}{\sqrt{\tau(1-\tau)}}. \] 
Setting $\dd{\mu_h(\tau)} = a_h(\tau) \dd{\tau}$ then yields 
\[ \int_0^1\frac{1}{\sqrt{\tau(1-\tau)}} \dd{|\mu_h|(\tau)} = \int_{\mathbb R}|h(s)| \dd{s} < \infty.\]

\subsubsection{Beyond parity masking} We conclude our exploration of $\cBwc$ by presenting an alternative mechanism to the one associated with the Born--Jordan concentration problem. The construction ultimately leverages Proposition \ref{prop:PositiveLocalizedAverage}, where we showed that, for positive operators, a local form of compactness (i.e., that of the mixed-state localization operator $H_{\Omega,S} = \chi_\Omega\star S=\int_\Omega \alpha_z(S)\dd{z}
$) suffices to ensure that $S \in \cBwtn$. This criterion is automatic in the case where $S \in \cK$, but it is not at all clear if noncompact positive operators with compact $H_{\Omega,S}$ do actually exist. Having in mind the previous section, where weak continuity is obtained by averaging parity, the main thrust here is to start with a noncompact operator whose localized averages become compact. The next result provides sufficient conditions for an explicit model. 

\begin{prop}\label{prop:PositiveDiagonalSeries}
Given an orthonormal sequence \((u_k)_{k\ge1}\) in \(L^2(\rd)\) and numbers \(0\le\lambda_k\le M\) for some fixed $M>0$, define (in the strong operator sense) 
\[S\coloneqq\sum_{k=1}^\infty\lambda_kP_k, \qquad P_k\coloneqq u_k\otimes u_k.\] 
Let $\Omega \subset \rdd$ be such that $0<|\Omega|<\infty$ and assume that 
\[ \sum_{k=1}^\infty\lambda_k\|H_{\Omega,k}\|_{\mathcal B}<\infty, \qquad H_{\Omega,k}\coloneqq\chi_\Omega\star P_k.\]
Then the map $f \mapsto Q_S f$ is weak-to-norm continuous $L^2(\rd) \to L^p(\Omega)$ for every $1 \le p < \infty$, hence \(S\in\cBwc\). Moreover, if in addition \(\lambda_k\not\to0\) then \(S\notin\cK\).
\end{prop}

\begin{proof}
We have \(0\le S\le M \mathrm{Id}\) by construction. If \(\lambda_k\not\to0\), then (possibly after passing to a subsequence) we have \(\lambda_k\ge\varepsilon\) for some $\varepsilon>0$. Since \(u_k\rightharpoonup0\), if \(S\) were compact then \(Su_k\to0\) strongly, but \(\|Su_k\|_{L^2} = \lambda_k\ge\varepsilon\), a contradiction. This proves that \(S\notin\cK\).

We clearly have \(H_{\Omega,k} \in \cS^1\), and thus
\[H_{\Omega,S}=\int_\Omega\alpha_z(S)\dd{z}=\sum_{k=1}^\infty\lambda_kH_{\Omega,k}\]
with convergence in operator norm by the summability assumption. Thus \(H_{\Omega,S}\) converges in operator norm as a sum of compact operators, hence \(H_{\Omega,S}\in\cK\). The claim then follows by Proposition \ref{prop:PositiveLocalizedAverage}.
\end{proof}

To summarize, although the operator \(S=\sum_{k=1}^\infty\lambda_k u_k\otimes u_k\) may fail to be compact if the eigenvalues \(\lambda_k\) do not tend to zero, a finite phase space region \(\Omega\) is ultimately only sensitive to the effect of the localized operator $H_{\Omega,S} = \int_\Omega\alpha_z(S)\dd{z}$.

It is clear from the previous argument that in the statement we assumed a summability condition that is strictly stronger than what would suffice to achieve the conclusion, that would be 
\[ \left\|\sum_{k>N}\lambda_kH_{\Omega,k}\right\|_{\mathcal B}\to 0, \] 
which is precisely equivalent to $H_{\Omega,S} \in \cK$. On the other hand, our practical sufficient condition can be more easily verified and generalized: For instance, it is easy to realize that the claim extends to more general operators \(S = \sum_{k=1}^\infty\lambda_k\Pi_k\) with finite-rank projections \(\Pi_k\) onto mutually orthogonal finite-dimensional subspaces \((E_k)_{k\ge1}\) of $L^2(\rd)$, provided that for every bounded measurable \(\Omega\subset\rdd\) one has
\[\sum_{k=1}^\infty\lambda_k\|H_{\Omega,\Pi_k}\|_{\cB}<\infty.\]

On the other hand, additional regularity could make the construction even more interesting. For instance, consider the stronger assumption
\[\sum_{k=1}^\infty\lambda_k C_\Omega(u_k)<\infty,\qquad C_\Omega(u_k)^2\coloneqq\int_\Omega\int_\Omega| A(u_k)(w-z)|^2\dd{z}\dd{w}.\] 
This ultimately amounts to asking Hilbert--Schmidt regularity of the pieces: 
\[\begin{aligned}\|H_{\Omega,k}\|_{\cS^2}^2&=\int_\Omega\int_\Omega\tr(\alpha_z(P_k)\alpha_w(P_k))\dd{z}\dd{w}\\&=\int_\Omega\int_\Omega|
A(u_k)(w-z)|^2\dd{z}\dd{w}=C_\Omega(u_k)^2,\end{aligned}\]
therefore 
\[\sum_{k=1}^\infty\lambda_k\|H_{\Omega,k}\|_{\cB}\le\sum_{k=1}^\infty\lambda_k C_\Omega(u_k)<\infty.\]
The main advantage of this formulation is that these conditions are expressed in terms of local norms of ambiguity functions \(A(u_k)(z)=\langle u_k,\pi(z)u_k\rangle\), and as soon as the \(A(u_k)\) lose their mass on a fixed bounded phase space region fast enough, the localized pieces \( H_{\Omega,k}\) can be summed to form a compact operator (even though the original diagonal series remains noncompact!). This phenomenon can be actually engineered in this setting. As a concrete illustration, note that for a fixed window \(\Omega\) the relevant quantity is $C_\Omega(u)$, and the summability assumption is slightly stronger than the local decay condition $C_\Omega(u_k)\to0$. If \(\Omega\) is bounded and \(\Omega-\Omega\subset B_R\), then we have
\[C_\Omega(u_k)\le|\Omega|^{1/2}\left(\int_{B_R}|A(u_k)(z)|^2\dd{z}\right)^{1/2},\]
so local decay on the single ball \(B_R\) is enough. Whenever such decay holds along a family \((v_n)\), one can choose a sufficiently sparse subsequence \(u_k=v_{n_k}\) so that the summability condition is enforced. Consider for instance a suitably sparsified family of Hermite functions: in that case we have
\[ A(\phi_n)(z)=e^{-\pi |z|^2/2}L_n(\pi |z|^2), \]
where $L_n$ is a Laguerre function, whose asymptotics give \(A(\phi_n)(z)\to0\) for $z \ne 0$, while \(|A(\phi_n)(z)|\le1\). As a consequence of dominated convergence we thus have 
\[\int_{B_R} |A(\phi_n)(z)|^2 \dd{z} \longrightarrow0 \quad \text{ for every ball } B_R\subset\mathbb R^2. \] It is then enough to choose a subsequence \(n_k\) so sparse that \(C_\Omega(\phi_{n_k})\le 2^{-k}\), and \(\lambda_k=1\).

\section{Concentration problems for operators} \label{sec:operators}
In this final section we will consider concentration problems related to phase space representations of operators. As for functions, representing an operator on phase space is useful in many situations. Consider for instance the Husimi function, defined for a bounded operator \(T\) as \[H_T(z)=\langle T\pi(z)\phi_0,\pi(z)\phi_0\rangle,\] which has a wide array of uses in quantum mechanics and related areas of physics and mathematics, see for instance~\cite{Schupp,FrankNicolaTilli,NicolaRiccardiTilli,bayer,CorderoGrochenig}. Similarly, many of the central objects in QHA can be interpreted as phase space representations of an operator, and we will consider here the associated concentration problems. The point of this section is to illustrate the correspondence between such representations of operators and an analogous representation on the operator's Weyl symbol. 

\subsection{Operator convolutions}
As we have seen, given an operator \(S\), we get a quadratic time-frequency distribution of the function \(f\) via convolution:\begin{align*}
    f\to (f\otimes f)\star \widecheck{S}.
\end{align*}
It is therefore natural to interpret the general convolution \begin{align*}
    T\to T\star \widecheck{S}
\end{align*}
as a time-frequency representation of the operator \(T\). This perspective is not only justified by the relation to Cohen's class, but also due to its connection with the Husimi function, which is readily recognized as a convolution:\begin{align*}
        H_{T}(z)=\langle T\pi(z)\phi_0,\pi(z)\phi_0\rangle=T\star\widecheck{\phi_0\otimes\phi_0}(z).
\end{align*}
As such, we can think of the distribution \(T\star\widecheck{S}\) as a generalized Husimi function. The study of such time-frequency representations was initiated by Klauder and Skagerstam~\cite{KS07}, and they have also been studied by Luef and Skrettingland~\cite{LSBerezin}. We will consider the following concentration problem for the generalized Husimi function: \begin{align}\label{OperatorProblem}
        \sup_{T\in\S^2\setminus\{0\}}\frac{\left(\int_{\Omega}|T\star\widecheck{S}(z)|^p \dd{z}\right)^{1/p} }{\|T\|_{\S^2}}.
\end{align}

The special case \(p=1\) is known to be connected to the so-called \textit{operator-valued localization operators}~\cite{QLS}. Using this correspondence, the case \(p=1\) has been considered in \cite{QLS,SpectralDeviation} using different techniques. In the operator setting we can also define a self-representation of \(T\) via\begin{align*}
    T\to T\star \widecheck{T}(z).
\end{align*}
This is known as the total correlation function \cite{LS20,LocalStruc}, which is denoted by \(\widetilde{T}(z)\), and gives us the problem \begin{align*}
        \sup_{T\in\S^2\setminus\{0\}}\frac{\left(\int_{\Omega}|\widetilde{T}(z)|^p \dd{z}\right)^{1/p} }{\|T\|_{\S^2}^2}.
\end{align*}
To settle these problems, let us recall a fundamental fact from QHA~\cite[Proposition 2.2]{LS20}: We have \(T\star S(z)=a_T*a_S(z)\), namely the convolution of operators is the function given by convolving the Weyl symbols. Thus, letting \(T_z\) denote translation by \(z\), the integrals in question are exactly equal to \begin{align*}
                \sup_{T\in\S^2\setminus\{0\}}\frac{\left(\int_{\Omega}|a_T * \widecheck{a_S}(z)|^p \dd{z}\right)^{1/p} }{\|T\|_{\S^2}}=        \sup_{T\in\S^2\setminus\{0\}}\frac{\left(\int_{\Omega}|\langle a_T,\overline{T_za_S}\rangle|^p \dd{z}\right)^{1/p} }{\|T\|_{\S^2}}.
\end{align*}
Existence of optimizers therefore boils down to the same problem for the function-function convolution of the symbols, which was studied before in \cite{Nicola-Romero-Trapasso-2022}. 

\begin{proof}[Proof of Proposition~\ref{Prop:HSExistence}]
    Let us treat the total correlation case first. We have \begin{align*}
         \sup_{T\in\S^2\setminus\{0\}}\frac{\left(\int_{\Omega}|\widetilde{T}(z)|^p \dd{z}\right)^{1/p} }{\|T\|_{\S^2}^2}&=\sup_{T\in\S^2\setminus\{0\}}\frac{\left(\int_{\Omega}|a_T*\widecheck{a_T}(z)|^p \dd{z}\right)^{1/p} }{\|T\|_{\S^2}^2}\\
         &=\sup_{a_T\in L^2(\R^{2d})\setminus\{0\}}\frac{\left(\int_{\Omega}|a_T*\widecheck{a_T}(z)|^p \dd{z}\right)^{1/p} }{\|a_T\|_{L^2}^2}\\
         &=\sup_{a_T\in L^2(\R^{2d})\setminus\{0\}}\frac{\left(\int_{\Omega}|\langle a_T,\overline{T_za_T}\rangle|^p \dd{z}\right)^{1/p} }{\|a_T\|_{L^2}^2},
    \end{align*}
where we can replace the norm of the operator with that of the symbol by Pool's theorem~\cite{Pool}. The same argument as in the proof of \cite[Proposition 1.3]{Nicola-Romero-Trapasso-2022} now applies, showing that the supremum is \(|\Omega|^{1/p} \), but it is not attained. 

For generalized Husimi case, optimizing over Hilbert--Schmidt operators reduces via Weyl symbols to a fixed-window convolution problem, which shares the same compactness mechanism underlying fixed-window STFT concentration studied in the proof of \cite[Proposition 5.1]{Nicola-Romero-Trapasso-2022}. Indeed, note that for any fixed function \(g\in L^2(\R^{2d})\) the functional \(f\to f*\widecheck{g}(z)=\langle f, \overline{T_zg}\rangle\) is sequentially weakly continuous on \(L^2(\R^{2d})\). Since we also have the \(L^\infty\)-estimate \(\|f*\widecheck{g}\|_{L^\infty}\leq \|f\|_{L^2}\|g\|_{L^2}\), the dominated convergence theorem implies that the functional \(f\mapsto f*\check g|_\Omega \) is weak-to-strong sequentially continuous from \(L^2(\R^{2d})\) to \(L^p(\Omega)\). By Pool's theorem, the same is true for the functional \(T\to T \star \widecheck{S}(z)\), and the claim follows by the direct method. 
\end{proof}

Before moving on, let us also briefly mention the $p=\infty$ case, as the solutions are quite simple. 
\begin{prop}\label{Prop:HSLInfinity}
    Let \(\Omega\subset \R^{2d}\) be measurable with \(0<|\Omega|\), and let \(S\in \S^2\). We have
    \begin{align*}
        \sup_{T\in\S^2\setminus\{0\}}\frac{\|T\star\widecheck{S}\|_{L^\infty(\Omega)}}{\|T\|_{\S^2}}=\|S\|_{\S^2},
\end{align*}
    and the supremum is attained. 
\end{prop}
\begin{proof}
    Consider from now on $S \ne 0$, otherwise the quotient is zero and in that case every $T \in \cS^2 \smo$ is an optimizer. 
    Young's inequality gives the estimate\begin{align*}
    \|T\star\widecheck{S}\|_{L^\infty(\Omega)}\leq\|T\star\widecheck{S}\|_{L^\infty(\R^{2d})}\leq\|T\|_{\S^2}\|\widecheck{S}\|_{\S^2}=\|T\|_{\S^2}\|S\|_{\S^2},
    \end{align*}
    which shows that \(\|S\|_{\S^2}\) is an upper bound. To obtain the lower bound, let \(w\) be a density point of \(\Omega\) and choose \(T=\alpha_w(S^*)\). Then \(\|T\|_{\cS^2}=\|S\|_{\cS^2}\) and 
    \[ T\star\widecheck S(w)=\operatorname{tr}(\alpha_w(S^*)\alpha_w(S))= \tr(S^*S) = \|S\|_{\S^2}^2. \] 
   Since the function \(z\mapsto (T\star\widecheck S)(z)\) is continuous, for every \(\varepsilon>0\) there is a neighborhood \(U\) of \(w\) such that \(U\cap\Omega\) has positive measure and 
\[ |(T\star\widecheck S)(z)|\ge \|S\|_{\cS^2}^2-\varepsilon, \qquad z\in U. \]
As a result, after division by \(\|T\|_{\cS^2}=\|S\|_{\cS^2}\) and letting \(\varepsilon\downarrow0\) we obtain
\[ \sup_{T \in \cS^2 \smo}\frac{\|T\star\widecheck S\|_{L^\infty(\Omega)}}{\|T\|_{\cS^2}}\ge \|S\|_{\cS^2}.\]
It is then clear that the supremum is attained, say by \(T=\alpha_w(S^*)\).
    \end{proof}

\begin{prop}\label{prop:RadarTC}
    Let \(T\in \S^2\) be a nonzero Hilbert--Schmidt operator, and \(z\in \R^{2d}\). Then
    \begin{align*}
        |\tilde{T}(z)|\leq \|T\|_{\S^2}^2
    \end{align*} 
    with equality if and only if \(z=0\) and \(T^*=cT\) for \(c\in \C\) with \(|c|=1\).

    As a consequence, for every measurable \(\Omega\subset \R^{2d}\) with \(|\Omega|>0\),
    \begin{align*}
        \sup_{T\in\S^2\setminus\{0\}}\frac{\|\widetilde{T}\|_{L^\infty(\Omega)}}{\|T\|_{\S^2}^2}=1.
\end{align*}
The supremum is attained if and only if \(|\Omega\cap B_r|>0\) for every \(r>0\). In this case, the optimizers are precisely the nonzero Hilbert--Schmidt operators satisfying \(T^*=cT\) for some \(c\in \C\) with \(|c|=1\).
\end{prop}

\begin{proof} The proof mimics the standard strategy for obtaining the classical radar correlation estimate. By Cauchy--Schwarz we have  
\[|\widetilde T(z)|=|\operatorname{tr}(T\alpha_z(T))|\le\|T\|_{\cS^2}\|\alpha_z(T)\|_{\cS^2}=\|T\|_{\cS^2}^2.\] 
Equality holds here if and only if \[\alpha_z(T)=cT^*\] for some unimodular \(c \in \C\). Taking adjoints gives \(\alpha_z(T^*)=\overline c\,T\), and repeated application of \(\alpha_z\) again yields \(\alpha_{2z}(T)=c\alpha_z(T^*)=c\overline c\,T=T\). Therefore, equality at \(z\) implies that \(T\) is invariant under \(\alpha_{2z}\). We claim that this is impossible for \(z\ne0\) and nonzero Hilbert--Schmidt \(T\). Indeed, taking the Fourier--Wigner transform gives 
\[\mathcal F_W(T)(w)=e^{-2\pi i[2z,w]}\mathcal F_W(T)(w) \implies (1-e^{-2\pi i[2z,w]})\mathcal F_W(T)(w)=0.\] 
Since \(z \ne 0\) and \(\mathcal F_W(T)\in L^2(\rdd)\), the latter must be supported on the set \(\{w \in \rdd :e^{-2\pi i[2z,w]}=1\}\), which is a countable union of affine hyperplanes of codimension one, hence it has Lebesgue measure zero. As such, we have $\mathcal F_W(T)=0$ a.e., and thus \(T=0\) by injectivity, a contradiction. Therefore, equality can occur only for \(z=0\), which is exactly \[T=cT^*\] for some \(|c|=1\), or equivalently \(T^*=\overline c\,T\) as claimed. 

Let us now prove the concentration statement. The upper bound follows immediately from \(|\widetilde T(z)|\le\|T\|_{\cS^2}^2\), just proved. For the converse, note that for rank-one operators \(T=f\otimes f\) one has \(\widetilde T(z)=(T\star\widecheck T)(z)=|Af(z)|^2\), hence the lower-bound construction of~\cite[Proposition 1.2]{Nicola-Romero-Trapasso-2022} applies here, and gives
\[ \sup_{T\in \cS^2\smo}\frac{\|\widetilde T\|_{L^\infty(\Omega)}}{\|T\|_{\cS^2}^2}\ge1, \]
hence the claim. 

Concerning attainment, suppose first that \(|\Omega\cap B_r|>0\) for every \(r>0\), and let \(T\ne0\) satisfy \(T^*=cT\). Then equality holds at \(z=0\), so \(|\widetilde T(0)|=\|T\|_{\cS^2}^2\). By continuity we have that \(|\widetilde T(z)|\) is arbitrarily close to \(\|T\|_{\cS^2}^2\) on small balls around \(0\), and each such ball intersects \(\Omega\) in positive measure, so \(\|\widetilde T\|_{L^\infty(\Omega)}=\|T\|_{\cS^2}^2\) and \(T\) is thus an optimizer. 

Conversely, suppose the supremum is attained by some \(T\ne0\); then \(\|\widetilde T\|_{L^\infty(\Omega)}=\|T\|_{\cS^2}^2\). We first show that \(|\Omega\cap B_r|>0\) for every \(r>0\). Otherwise, \(|\Omega\cap B_r|=0\) for some \(r>0\), so the \(L^\infty(\Omega)\)-norm is unchanged if \(\Omega\) is replaced by \(\Omega\setminus B_r\), which is contained in \(\{|z|\ge r\}\). By the previous equality statement, strict inequality must hold at every point with \(|z|\ge r\). Since \(\widetilde T\in C_0(\rdd)\), as seen by an easy finite-rank approximation argument, the supremum of \(|\widetilde T|\) on the closed set \(\{|z|\ge r\}\) is attained, and is strictly smaller than \(\|T\|_{\cS^2}^2\), which is a contradiction. It remains to identify the optimizers: Again due to \(\widetilde T\in C_0(\rdd)\), for each \(n\) the set 
\[E_n\coloneqq\left\{z\in\Omega:|\widetilde T(z)|>\|T\|_{\cS^2}^2-\frac1n\right\}\]
has positive measure. The \(C_0\)-property gives \(R>0\) such that \(|\widetilde T(z)|<\frac12\|T\|_{\cS^2}^2\) for \(|z|\ge R\), so for \(n\) large enough we have \(E_n\subset B_R\) up to a null set. We can thus choose \(z_n\in E_n\cap B_R\) and, after passing to a subsequence, assume \(z_n\to z_*\). Continuity then implies \(|\widetilde T(z_*)|=\|T\|_{\cS^2}^2\). The equality statement implies \(z_*=0\) and \(T^*=cT\) for some \(|c|=1\), therefore the optimizers are precisely the nonzero Hilbert--Schmidt operators satisfying \(T^*=cT\).
\end{proof}

\begin{remark}
    Unlike the classical radar estimate, equality cannot be attained for arbitrary inputs \(S\). This is due to the occurrence of nilpotent elements in \(\S^2\). If for instance \(f,g \in L^2(\R^d)\) are two orthonormal functions, the Hilbert--Schmidt operator \(f\otimes g\) has norm \(1\), but the total correlation at \(0\) is \(\tr(f\otimes g f\otimes g)=|\langle f,g\rangle|^2=0\). Since \(z=0\) is required to attain the maximal value, the total correlation cannot be \(1\) at any other point, either.
\end{remark}

\subsubsection{Beyond Hilbert--Schmidt} So far in this section we have optimized over the Hilbert--Schmidt operators. This is the most natural analogue of the supremum in Problem \eqref{Problem}, but it is not difficult to enlarge the input regularity to all the Schatten classes, at least for the generalized Husimi concentration problem. 

\begin{prop} \label{prop:SchattenLinearHusimi}
Consider \(1\le p<\infty\), \(\Omega\subset\rdd\) with \(0<|\Omega|<\infty\), and \(S\ne0\).

\begin{enumerate}
\item If \(1<q<\infty\), \(q'\) is conjugate to \(q\), and \(S\in\cS^{q'}\), then the map
$ T\mapsto T\star\widecheck S $ is sequentially weak-to-norm continuous from bounded subsets of \(\cS^q\) into \(L^p(\Omega)\). Consequently, the supremum
\[ \sup_{T\in\cS^q\setminus\{0\}} \frac{\|T\star\widecheck S\|_{L^p(\Omega)}}{\|T\|_{\cS^q}} \]
is attained.

\item If \(q=1\) and \(S\in \cK\), then the same conclusion holds in the weak* sense with input \(\cS^1\).

\item If \(q=\infty\) and \(S\in \cS^1\), then the same conclusion holds in the weak* sense with input \(\cB\). 
\end{enumerate}
\end{prop}

\begin{proof}
For every admissible input \(T\) one has $ (T\star\widecheck S)(z)=\operatorname{tr}(T\alpha_z(S))$.

Assume first that \(1<q<\infty\), \(S\in\cS^{q'}\), and \(T_n\rightharpoonup T\) weakly in \(\cS^q\), with \((T_n)\) bounded. Then, for every fixed \(z\in\rdd\), Schatten duality and \(\alpha_z(S)\in\cS^{q'}=(\cS^q)^*\) yield
\[ \tr(T_n\alpha_z(S)) \to \tr(T\alpha_z(S)). \]
If \(\|T_n\|_{\cS^q}\le M\) for some $M>0$, then \(\|T\|_{\cS^q}\le M\) as well and
\[ |(T_n-T)\star\widecheck S(z)| \le 2M\|S\|_{\cS^{q'}}. \]
Since \(|\Omega|<\infty\), dominated convergence yields 
\[ \|T_n\star\widecheck S-T\star\widecheck S\|_{L^p(\Omega)}\to0, \]
proving that the map \(T\mapsto T\star\widecheck S\) is weak-to-norm continuous on bounded subsets as claimed. As a result, $T\mapsto \|T\star\widecheck S\|_{L^p(\Omega)}$ is weakly continuous over the (weakly compact) closed unit ball of \(\cS^q\). 

The remaining cases can be handled similarly, with minimal updates, recalling the identifications $\cS^1=\cK^*$ and $\cB = (\cS^1)^*$.
\end{proof}

We mention that the Hilbert--Schmidt assumption can be even varied beyond the Schatten class regularity. For instance, one may keep the Weyl-symbol point of view and replace the symbol space \(L^2(\rdd)\) by a modulation space, and define \textit{operator modulation spaces}~\cite{QTFA}: 
\[ \mathcal M^{r} \coloneqq \operatorname{op}^{\mathrm{w}}\bigl(M^{r}(\rdd)\bigr),
\qquad \|T\|_{\mathfrak M^{r}} \coloneqq \|a_T\|_{M^{r}}, \] 
where \(a_T\) is the Weyl symbol of \(T\). We prefer to avoid this level of generality here, which comes with additional technicalities. 

The total correlation problem cannot be extended with the same ease beyond $\cS^2$. For \(1\le q<2\), the same approach only gives the universal bound
\[ \|T\star\widecheck T\|_{L^p(\Omega)} \le |\Omega|^{1/p}\|T\|_{\cS^2}^2 \le |\Omega|^{1/p}\|T\|_{\cS^q}^2, \]
thus the quotient is bounded by \(|\Omega|^{1/p}\). However, sharpness of this bound would force any normalized maximizing sequence \(\|T_n\|_{\cS^q}=1\) to satisfy $\|T_n\|_{\cS^2}\to1$, namely the singular values of \(T_n\) must asymptotically saturate the embedding \(\cS^q\subset\cS^2\) and concentrate in an essentially rank-one profile, where in case \(\Omega\) is a ball we know the value of the supremum will smaller~\cite{Flandrin}. For \(q>2\) the situation is even worse, since the basic \(\cS^2\)-based estimate is unavailable without extra regularity, since \(\cS^q\not\subset\cS^2\). To summarize, the general Schatten-input total-correlation problem is not resolved by the Hilbert--Schmidt method.

\subsubsection{Optimization over density operators} Let us now consider the operator optimization problem for a smaller class of operators, namely the density operators. We define the density operators to be the set \[\D=\{S\in \S^1\colon S\geq 0, \tr(S)=1\}.\] There are several reasons why one would consider optimization over density operators. In light of Young's theorem for operator-operator convolutions, this allows us to extend the window class from \(\S^2\) to \(\mathcal{K}\). More importantly, the density operators are central to quantum mechanics, as they describe the state of a quantum system~\cite[Chapter 13]{deGossonBook}, and in that sense optimization over density operators can be thought of as a more ``physical'' problem than optimization over Hilbert--Schmidt operators. Indeed, the Husimi function is usually only considered for \(T\in \D\). Since the set of density operators is convex, the conclusion of Proposition~\ref{DensityOptimization} is to be expected. Nevertheless, some care must be taken in order to prove the result.

\begin{proof}[Proof of Proposition \ref{DensityOptimization}]
    Let \(T\in\D\) and write the spectral decomposition
    \[  T=\sum_{n=1}^{\infty}\lambda_n f_n\otimes f_n,\qquad \lambda_n\ge0,\qquad \sum_n\lambda_n=1, \qquad \|f_n\|_{L^2}=1. \]
    The spectral expansion converges in $\cS^1$, and by linearity and uniform boundedness of $\alpha_z(S)$ with respect to $z$ we have that \( T\star\widecheck S=\sum_{n=1}^{\infty}\lambda_n Q_Sf_n\) pointwise. Dominated convergence and Minkowski's inequality yield
    \[  \|T\star\widecheck S\|_{L^p(\Omega)} \le \sum_{n=1}^{\infty}\lambda_n\|Q_Sf_n\|_{L^p(\Omega)} \le \Lambda_{p,\Omega}(S). \]
    The reverse inequality follows at once because every rank-one density \(f\otimes f\) with \(\|f\|_{L^2}=1\) belongs to \(\D\). Therefore, we have proved that
    \[ \sup_{T\in\D}\|T\star\widecheck S\|_{L^p(\Omega)}=\Lambda_{p,\Omega}(S). \]
    If \(\Lambda_{p,\Omega}(S)\) is attained by a unit vector \(f\), then \(T=f\otimes f\) is a maximizing density operator. In particular, this applies to \(S \in \cK\) by Theorem~\ref{CompactExistence}.\end{proof}

\begin{remark}
As an alternative route, one can introduce the auxiliary set
\[
\mathcal C\coloneqq\{T\in\cS^1:T\ge0,\ \operatorname{tr}(T)\le1\}
\]
as a weak-* compact convex replacement for \(\D\). One can then prove that its extreme points are
\[ \operatorname{ex}(\mathcal C)=\{0\}\cup\{f\otimes f:\|f\|_{L^2}=1\}
=\{0\}\cup\operatorname{ex}(\mathcal D). \]
Thus, in principle, Bauer's maximum principle~\cite{Bauer} can be applied to \(\mathcal C\), provided one verifies the required weak-star continuity or upper semicontinuity of the functional. 
\end{remark} 

The main takeaway from Proposition \ref{DensityOptimization} should be that when optimizing over the density operators, all the optimization properties of the distribution \(T\star\widecheck{S}\) can be deduced by looking at the Cohen class distribution \(Q_Sf\) instead. In the case \(p=1\) and a positive \(S\) the same conclusion was reached in \cite{QLS,SpectralDeviation}. There it was also shown that the corresponding localization operators share the same eigenvalues, and as such, when \(S\) is positive and \(p=1\) the equality in Proposition \ref{DensityOptimization} extends to all stationary points of the problem, not just the optimizer.

Lastly, in the \(L^\infty\)-case one might think one can do better than \(\Lambda_{\infty,\Omega}(S)\) when optimizing over \(\D\), since there are more operators to choose from. This is not the case: The conclusion of Proposition~\ref{DensityOptimization} also holds in the case \(p=\infty\).

\begin{prop}\label{Prop:DensityLInfinity}
Let \(\Omega\subset\R^{2d}\) be such that $0<\abs{\Omega}$ and let $S \in \cB$. Then
\[
        \sup_{T\in\D}\|T\star\widecheck{S}\|_{L^\infty(\Omega)}=w(S)=\Lambda_{\infty,\Omega}(S),
\]
and the supremum is therefore attained if \(S\in \NRA\).
\end{prop}
\begin{proof}
We argue again by spectral decomposition: Since \(T\ge0\) and \(\tr(T)=1\), we can write
\[ T=\sum_j\lambda_j f_j\otimes f_j, \qquad \lambda_j\ge0,\qquad \sum_j\lambda_j=1, \]
with convergence in $\cS^1$. Hence, for every \(z\in\R^{2d}\), linearity yields $(T\star\widecheck S)(z) = \sum_j\lambda_j Q_Sf_j(z)$, therefore
\[ |(T\star\widecheck S)(z)| \le \sum_j\lambda_j |Q_Sf_j(z)| =  \sum_j\lambda_j| \langle\alpha_z(S)f_j,f_j\rangle| \le  \sum_j\lambda_j w(\alpha_z(S)). \]
Since \(w(\alpha_z(S))=w(S)\), we have \( |(T\star\widecheck S)(z)|\le w(S)\), hence \( \sup_{T\in\D}\|T\star\widecheck S\|_{L^\infty(\Omega)}\le w(S).\) 

For the converse inequality, it again suffices to test over rank-one density operators. If \(\|f\|_{L^2}=1\), then \(f\otimes f\in\D\), and $(f\otimes f)\star\widecheck S=Q_Sf$, therefore 
\[ \sup_{T\in\D}\|T\star\widecheck S\|_{L^\infty(\Omega)} \ge \sup_{\|f\|_{L^2}=1}\|Q_Sf\|_{L^\infty(\Omega)} = 
\Lambda_{\infty,\Omega}(S), \]
and the claim follows since $\Lambda_{\infty,\Omega}(S)=w(S)$.
\end{proof}

\subsection{Polarized Cohen's class}
Lastly, we consider a concentration problem when an operator \(T\) is represented on double phase space. This is done using the polarized Cohen's class, first studied in \cite{QTFA}. Like with the convolution representation, we will see that all the information in the representation is given in terms of the Weyl symbols.

The polarized Cohen's class is a representation of an operator \(T\) which is reminiscent of the short-time Fourier transform for functions. Its construction starts by polarizing the operator shift \(\alpha_z(S)=\pi(z)S\pi(z)^*\), which produces the \(\gamma\)-shifts\begin{align*}
    \gamma_{w,z}(S)=\pi(z)S\pi(w)^*.
\end{align*}
The \(\gamma\)-shift is a projective representation on the Hilbert--Schmidt operators, unlike the operator shift \(\alpha\), which is a unitary representation on the Hilbert--Schmidt operators. On the other hand, the time-frequency shift \(\pi\) is a projective representation on \(L^2\), so the idea is that the \(\gamma\)-shift on \(\S^2\) behaves like the time-frequency shifts on \(L^2\). With this in mind, we define the \textit{polarized Cohen's class} of \(T\) with respect to \(S\) as the function\begin{align*}
    Q_ST(w,z)=\langle T, \gamma_{w,z}(S)\rangle_{\S^2},
\end{align*} 
where \(S,T\in \S^2\) and \(w,z\in \R^{2d}\). Since the \(\gamma\)-shift is an operator analogue of the time-frequency shift, it is natural to believe that the polarized Cohen's class should inherit many of the properties of the STFT, and this is indeed the case. While we will not make use of them, we list some properties of the polarized Cohen's class, and direct the reader to Section 3 of \cite{QTFA} for details.
\begin{theorem}
    Let \(R,S,T,W\in \S^2\).
    \begin{itemize}
        \item $\langle Q_RS,Q_TW\rangle_{L^2(\R^{4d})}=\langle R,T\rangle_{\S^2}\overline{\langle S,W\rangle_{\S^2}}.$
        \item If \(\|S\|_{\S^2}=1\) then $T\to Q_ST$ is an isometry from \(\S^2\) to \(L^2(\R^{4d})\).
        \item If \(\|S\|_{\S^2}=1\) then \(Q_S^*Q_ST=T\).
    \end{itemize}
\end{theorem}

The properties above really suggest that the polarized Cohen's class behaves like an operator analogue of the STFT. It is therefore natural to study the corresponding concentration problem. Given an operator \(S\in \S^2\), \(p\in [1,\infty)\) and a set \(\Delta\subset\R^{4d}\) in double phase space with \(0<\abs{\Delta}<\infty\), we consider the supremum
\begin{align*}
    \sup_{T\in \S^2\setminus\{0\}}\frac{\left(\int_{\Delta}|Q_ST(w,z)|^p \dd{z} \dd{w}\right)^{1/p} }{\|T\|_{\S^2}}.
\end{align*}

We will reduce this problem to a concentration problem for the regular STFT. The key observation is the following.
\begin{corollary}[Corollary 3.15 of \cite{QTFA}]
    Let \(S,T\in \S^2\) and \(z,w\in \R^{2d}\). Then\begin{align}\label{Problem4}
        |Q_ST(z,w)|=|V_{a_S}a_T(U(w,z))|,
    \end{align}
    where \(U\) is the linear change of variables given by the matrix \begin{align*}
        U=\begin{pmatrix}
            O_d & -I_d & O_d & I_d \\ I_d & O_d & -I_d & O_d \\ \dfrac{1}{2}I_d & O_d & \dfrac{1}{2}I_d & O_d \\ O_d & \frac{1}{2}I_d & O_d & \frac{1}{2}I_d
        \end{pmatrix}, \qquad \det(U)=1. 
    \end{align*}
\end{corollary}

The above corollary implies that the polarized Cohen's class is the STFT of the Weyl symbols, at least up to a phase factor and a change of variables. Pool's theorem therefore gives the following equivalence:\begin{align*}
     \sup_{T\in \S^2\setminus\{0\}}\frac{\left(\int_{\Delta}|Q_ST(w,z)|^p \dd{z} \dd{w}\right)^{1/p} }{\|T\|_{\S^2}}
    &=\sup_{T\in \S^2\setminus\{0\}}\frac{\left(\int_{\Delta}|V_{a_S}a_T(U(w,z))|^p \dd{z} \dd{w}\right)^{1/p} }{\|T\|_{\S^2}}\\
    &=\sup_{f\in L^2(\R^{2d})\setminus\{0\}}\frac{\left(\int_{\Delta}|V_{a_S}f(U(w,z))|^p \dd{z} \dd{w}\right)^{1/p} }{\|f\|_{L^2}}\\
    &=\sup_{f\in L^2(\R^{2d})\setminus\{0\}}\frac{\left(\int_{U(\Delta)}|V_{a_S}f(w,z)|^p \dd{z} \dd{w}\right)^{1/p} }{\|f\|_{L^2}}.
\end{align*}

By this equivalence, the existence of an optimizer to the polarized concentration problem is equivalent to the existence of an optimizer for the corresponding problem for the STFT concentration functional with fixed window $a_S$. We may now invoke the weak continuity of this final functional (see for instance \cite[Proposition 5.1]{Nicola-Romero-Trapasso-2022}) to conclude about the existence of optimizers.

\begin{corollary}
    Let \(p\in[1,\infty)\), \(\Delta\subset \R^{4d}\) be measurable with $0<|\Delta|<\infty$, and assume \(S\in \S^2\). Then the supremum
    \begin{align*}
        \sup_{T\in \S^2\setminus\{0\}}\frac{\left(\int_{\Delta }|Q_ST(w,z)|^p \dd{z} \dd{w}\right)^{1/p}}{\|T\|_{\S^2}}
    \end{align*}
    is attained. Furthermore, if $S \ne 0$ then any normalized maximizing sequence has a subsequence converging to a maximizer in \(\S^2\).
\end{corollary}

As a concluding remark, let us also note that since the two problems are equivalent, any result about the STFT concentration problem carries over to a result on the polarized Cohen's class. Consequently, results like Daubechies' theorem~\cite{Daubechies} and the Faber-Krahn inequality for the STFT~\cite{FaberKrahn} also hold true for the polarized Cohen's class, at least up to the volume-preserving change of variables \(U\).

\appendix 
	\section{Strict-gap criteria for the Wigner window} \label{app-wig}
    
In this appendix we record some situations in which the Wigner window \(S_{\mathrm{W}}=2^dP\) belongs to the strict-gap class \(\cG\). Throughout the section we assume as usual \(1\le p<\infty\), \(0<\abs{\Omega}<\infty\), with \(W(f,g)(z) = Q_{S_{\mathrm{W}}}(f,g)(z) \) and \(Wf\coloneqq W(f,f)\). 
Before entering the details, let us first emphasize that belonging to \(\cG\) rules out the occurrence of exotic optimizers in the sense of \cite{Stra-Svela-Trapasso-2025}, since no normalized weakly null sequence can be maximizing for \(\Lambda_{p,\Omega}(S_{\mathrm W})\). Indeed, by definition of the essential threshold, 
\[\limsup_{n\to\infty}\|Wf_n\|_{L^p(\Omega)}\le \Lambda^{\mathrm{ess}}_{p,\Omega}(S_{\mathrm W})<\Lambda_{p,\Omega}(S_{\mathrm W})\]
for every sequence \(f_n\rightharpoonup0\) with \(\|f_n\|_{L^2}=1\). As a consequence, the exotic escaping profiles possibly occurring in the Wigner concentration-compactness analysis carried out in \cite{Stra-Svela-Trapasso-2025} cannot attain the optimal value in the strict-gap regime. Equivalently, every maximizing sequence must contain at least one nontrivial compact profile.

Define \[M_{p,\Omega}(S_{\mathrm{W}})\coloneqq \sup_{\norm{u}_{L^2}=\norm{v}_{L^2}=1}\norm{W(u,v)}_{L^p(\Omega)}\] and \[C_p\coloneqq \left(\frac1{2\pi}\int_0^{2\pi}\abs{\cos\theta}^p\dd{\theta}\right)^{1/p}.\] 
Note that \(0<C_p<1\) for every finite \(p\). The main asymptotic estimates from \cite{Stra-Svela-Trapasso-2025} can be summarized as follows:
\[\Lambda^{\mathrm{ess}}_{p,\Omega}(S_{\mathrm{W}})\le C_pM_{p,\Omega}(S_{\mathrm{W}}).\] 
To be precise, this is the essential-threshold bound obtained by applying the profile decomposition of \cite[Lemma~2.3]{Stra-Svela-Trapasso-2025} to a normalized weakly null sequence, using the cross-Wigner covariance formula \cite[Lemma~2.1]{Stra-Svela-Trapasso-2025}, and then applying the antipodal cross-term asymptotic \cite[Theorem~3.1]{Stra-Svela-Trapasso-2025}. 

The strict gap condition is then 
\[C_pM_{p,\Omega}(S_{\mathrm{W}})<\Lambda_{p,\Omega}(S_{\mathrm{W}}) \implies S_{\mathrm{W}}\in\cG. \]
Note that, since the diagonal choice \(u=v\) is allowed in the definition of \(M_{p,\Omega}(S_{\mathrm{W}})\), one always has \(M_{p,\Omega}(S_{\mathrm{W}})\ge \Lambda_{p,\Omega}(S_{\mathrm{W}})\). The simple sufficient condition \(M_{p,\Omega}(S_{\mathrm{W}})\le \Lambda_{p,\Omega}(S_{\mathrm{W}})\) thus coincides with the equality \(M_{p,\Omega}(S_{\mathrm{W}})=\Lambda_{p,\Omega}(S_{\mathrm{W}})\).

We also record that, combining the elementary $L^\infty$ bound for the Wigner transform with Lieb's $L^p$ uncertainty for $p\ge2$, we immediately obtain
\begin{equation}\label{eq:PointwiseAndLieb}
    M_{p,\Omega}(S_{\mathrm{W}})\le 2^d\min\{\abs{\Omega}^{1/p},(2p)^{-d/p}\}, 
\end{equation}
where we computed $L_{p,d}(S_{\mathrm{W}})\coloneqq \norm{W\varphi_0}_{L^p}=2^d(2p)^{-d/p}$ and \(\varphi_0\) is the usual normalized Gaussian, giving $|W\varphi_0(z)|=2^d e^{-2\pi |z|^2}$. 

Let us start with a simple consequence of continuity for sufficiently small sets. 
\begin{prop}\label{prop:wigner-small-sets}
Let \(1\le p<\infty\), \(z_0\in\mathbb R^{2d}\), and choose \(a>0\) satisfying \(C_p<a<1\). Then there exists \(r>0\) such that, whenever \(0<\abs{\Omega}<\infty\) and \(\Omega\subset B(z_0,r)\), one has $S_{\mathrm W}\in\mathcal G(p,\Omega)$. 
\end{prop}

\begin{proof}
Choose a normalized eigenvector \(f\) of the displaced parity \(\alpha_{z_0}(P)\) with eigenvalue \(1\), \(\pi(z)f\) for any normalized even \(f\) will do, so that $Wf(z_0)=2^d$. Since \(z\mapsto Wf(z)\) is continuous, there exists \(r>0\) such that $\abs{Wf(z)}\ge 2^da$ for all $z\in B(z_0,r)$. If \(\Omega\subset B(z_0,r)\), then
\[\Lambda_{p,\Omega}(S_{\mathrm W})\ge 2^da\abs{\Omega}^{1/p}>2^dC_p\abs{\Omega}^{1/p}\ge C_pM_{p,\Omega}(S_{\mathrm W}), \]
and the claim follows. 
\end{proof}

The gap criterion can be enforced also by imposing bound on the captured Gaussian mass. 
\begin{prop}\label{prop:wigner-gaussian-mass}
Let \(p\ge2\) and set
\[I_\Omega\coloneqq \int_\Omega \abs{W\varphi_0(z)}^p\dd{z}.\]
 If \[\max\left\{\frac{I_\Omega}{2^{dp}\abs{\Omega}},\frac{I_\Omega}{2^{dp}(2p)^{-d}}\right\}>C_p^p,\] then \(S_{\mathrm{W}}\in\cG\).
\end{prop}

\begin{proof}
The assumption is equivalent to
\[ I_\Omega > C_p^p\min\{2^{dp}\abs{\Omega},2^{dp}(2p)^{-d}\}, \] therefore
\[ \Lambda_{p,\Omega}(S_{\mathrm W})^p \ge I_\Omega > C_p^pM_{p,\Omega}(S_{\mathrm W})^p. \]
Taking \(p\)-th roots yields $\Lambda_{p,\Omega}(S_{\mathrm W}) > C_pM_{p,\Omega}(S_{\mathrm W})$ so the strict-gap criterion allows one to conclude. 

\end{proof}

Finally, let us consider now the ball \(B_R=B(0,R)\subset\R^{2d}\). After setting $ x=2\pi pR^2$, define
\[A_d(x) \coloneqq \frac{d\,\gamma(d,x)}{x^d}, \qquad F_d(x) \coloneqq \frac{\gamma(d,x)}{\Gamma(d)},\]
where $\gamma(d,x)=\int_0^x t^{d-1}e^{-t} \dd{t}$. We thus infer 
\[ \|W\varphi_0\|_{L^p(B_R)}^p = 2^{dp}|B_R|A_d(x) = 2^{dp}(2p)^{-d}F_d(x). \]
Moreover, note that 
\[ A_d(x)\to1\quad\text{as }x\downarrow0, \qquad F_d(x)\to1\quad\text{as }x\to\infty, \]
and a straightforward computation shows that
\[m_d\coloneqq \inf_{x>0}\max\{A_d(x),F_d(x)\}=F_d((d!)^{1/d})=\frac{\gamma(d,(d!)^{1/d})}{\Gamma(d)}.\]
We have enough material to show that $S_\mathrm{W}$ is in the gap class at least for sufficiently small or large balls. 

\begin{prop}\label{prop:cpmd}
For \(d\ge1\) and \(p\ge2\), we have
\[C_p^p<m_d \implies S_{\mathrm W}\in\mathcal G(p,B_R) \quad \forall R>0.\]
\end{prop}

\begin{proof}
Let us start by setting $U_R=2^{dp}\abs{B_R},$ $V=2^{dp}(2p)^{-d}$ for simplicity, so that \(U_RA_d(x)=VF_d(x)\). As a result,
\[ \Lambda_{p,B_R}(S_{\mathrm W})^p \ge  U_RA_d(x) = VF_d(x), \]
while Equation~\eqref{eq:PointwiseAndLieb} yields $M_{p,B_R}(S_{\mathrm W})^p\le \min\{U_R,V\}$. If \(C_p^p<m_d\), after setting \(c=C_p^p\) we obtain
\[\max\{A_d(x),F_d(x)\}\ge m_d>c,\]
so at least one of \(A_d(x)\) and \(F_d(x)\) is strictly larger than \(c\). If \(A_d(x)>c\) then
\(U_RA_d(x)>cU_R\ge c\min\{U_R,V\}\), whereas if \(F_d(x)>c\) then \(VF_d(x)>cV\ge c\min\{U_R,V\}\). Since \(U_RA_d(x)=VF_d(x)\), in either case we obtain
\[\Lambda_{p,B_R}(S_{\mathrm W})^p>C_p^p\min\{U_R,V\}\ge C_p^pM_{p,B_R}(S_{\mathrm W})^p.\]
The claim follows by the strict-gap criterion after taking \(p\)-th roots. 

\end{proof}

\begin{corollary}\label{cor:wigner-collected-cases}
The following results hold for the Wigner window \(S_{\mathrm W}=2^dP\).

\begin{enumerate}

\item For every \(d\ge1\) and \(p\ge2\), one has \(S_{\mathrm{W}}\in\mathcal G(p,B_R)\) for all sufficiently small and large \(R>0\).

\item For every fixed \(d\ge1\), there exists \(p_d<\infty\) such that whenever \(p\ge2\) and \(p>p_d\), \[p > \max\{2,p_d\} \implies S_{\mathrm{W}}\in\mathcal G(p,B_R) \quad \forall R>0.\] 
The same conclusion holds for every translated symplectic ball, i.e. images of the ball under linear symplectic maps, 
\[z_0+A(B_R)\qquad \text{with }\quad  z_0\in\rdd,\quad A\in\mathrm{Sp}(2d,\mathbb R).\]

\end{enumerate}
\end{corollary}

\begin{proof}

For the first item, we resort to the Gaussian mass criterion in Proposition \ref{prop:wigner-gaussian-mass}. Recall the asymptotic behaviors
\[ A_d(x)\to1\quad (x\downarrow0), \qquad F_d(x)\to1\quad (x\to\infty). \]
If \(R\downarrow0\), then \(x\downarrow0\). Since \(C_p<1\), for all sufficiently small \(R\) we have $A_d(x)>C_p^p$, hence 
\[ \frac{\|W\varphi_0\|_{L^p(B_R)}^p}{2^{dp}|B_R|} = A_d(x) > C_p^p. \]

If \(R\to\infty\), then \(x\to\infty\). Again since \(C_p<1\), for all sufficiently large \(R\) we have $F_d(x)>C_p^p$, therefore 
\[ \frac{\|W\varphi_0\|_{L^p(B_R)}^p}{2^{dp}(2p)^{-d}}  = F_d(x) > C_p^p.\]

For the second item, note that 
\[ C_p^p = \frac1{2\pi}\int_0^{2\pi}|\cos\theta|^p\,d\theta = \frac{\Gamma((p+1)/2)}{\sqrt\pi\,\Gamma((p+2)/2)} \to0 \qquad  (p\to\infty). 
\] Therefore, there exists \(p_d<\infty\) such that $C_p^p<m_d$ for every \(p>p_d\), and thus the claim follows by Proposition~\ref{prop:cpmd}.

Finally, the passage from balls to translated symplectic balls is a consequence of the Wigner distribution being covariant under time frequency shifts and metaplectic operators, that is $\abs{W(\pi(z_0)f)(z)}=\abs{Wf(z-z_0)}$ and $\abs{W(\mu(A)f)(z)}=\abs{Wf(A^{-1}z)}$. In particular, since \(A\) is volume preserving and both \(\pi(z_0)\) and \(\mu(A)\) are unitary, these transformations preserve normalized weak convergence, and thus both \(\Lambda_{p,\Omega}(S_{\mathrm W})\) and \(\Lambda^{\mathrm{ess}}_{p,\Omega}(S_{\mathrm W})\) are invariant under replacing \(\Omega\) by \(z_0+A(\Omega)\). 
\end{proof}

\section{Concentration problems beyond the Heisenberg representation} \label{app-eucl}

In this appendix we elaborate on the direction first highlighted in Remark \ref{rem-eucl}, and slightly extend our analysis beyond the standard Euclidean setting. 

Throughout this section we consider a second countable locally compact group \(\Xi\) equipped with a left Haar measure \(\mu_\Xi\), while \(\Omega\subset\Xi\) is a measurable set with \(0<\mu_\Xi(\Omega)<\infty\). \(\cH\) is a complex separable Hilbert space, while \(\rho \colon \Xi\to\mathcal U(\cH) \) is a strongly continuous unitary representation. Given \(\xi\in\Xi\), we write \( \alpha_\xi(S)\coloneqq \rho(\xi)S\rho(\xi)^* \) for the induced action on \( \cB(\cH)\) --- note that it is not a problem to consider projective representations, since they cause no ambiguity in the definition of \(\alpha_\xi(S)\). Here \(\Xi\) plays the role of phase space or parameter space: In the ordinary locally compact abelian time-frequency setting over a signal group \(G_0\), one has \(\Xi=G_0\times\widehat{G_0}\), while in the Euclidean case \(G_0=\R^d\) we recover \(\Xi=\R^{2d}\). On the other hand, in the affine wavelet setting we have \(\Xi=\Xi_{\mathrm{aff}}=\R\rtimes\R_+\) (see below). 

In Remark \ref{rem-eucl} we have already mentioned that the arguments in the proof of Theorem \ref{CompactExistence} extend to the more general setting where the Cohen-type transform with operator window \(S \in \cB(\cH)\) is defined by
\[ Q_S f(\xi)\coloneqq \langle \alpha_\xi(S)f,f\rangle_{\cH}, \qquad f \in \cH. \] 
In fact, with obvious adaptation of the notation introduced in Section \ref{sec-structure}, most of the structural classification is robust beyond the Euclidean Heisenberg setting. For $1 \le p < \infty$ we have indeed
\[ \cK(\cH) =  \cBpt \subseteq \cBwtn = \cBwc =  \ker\Lambda^{\mathrm{ess}}_{p,\Omega} \subseteq \cBsup, \]
and the strict-gap criterion
\[ \Lambda_{p,\Omega}(S)>\Lambda^{\mathrm{ess}}_{p,\Omega}(S) \quad \implies \quad S \in \cBsup. \]
The whole $L^\infty$ concentration analysis extends as well, in particular we have
\[ \Lambda_{\infty,\Omega}(S)=w(S) \quad \forall\ S \in \cB(\cH), \qquad \NRA(\cH) \subseteq \cB_{\mathrm{sup}}(\infty,\Omega).\]
Let us here point out that the Wigner transform is inherently Heisenberg-dependent, since at this level of generality there is no canonical notion of parity unless the representation is symmetric under phase-space inversion. This is another argument in support of a QHA approach, since studying the Cohen class distributions from this angle does not suffer from such limitations.

It is therefore interesting to further investigate the negative results, where one expects a strong dependence on the representation. Having in mind the short-time Fourier transform and the ambiguity setting in \cite{Nicola-Romero-Trapasso-2022}, in line with the spirit of coorbit theory \cite{coorbit1,coorbit2} we introduce here the \textit{voice transform} of $f \in \cH$ with window $g \in \cH\smo$: 
\[ V_g f(\xi)\coloneqq \langle f,\rho(\xi)g\rangle_{\cH},\qquad \xi\in\Xi. \]
It is straightforward to show that the correspondence $f \mapsto V_gf$ is weak-to-norm continuous from bounded sets of $\cH$ to $L^p(\Omega)$, so concentration optimizers do exist by the direct method. The same argument turns out to readily extend to fixed-window generalized Husimi functions of the form
\[ C_R T(\xi)\coloneqq \langle T,\alpha_\xi(R)\rangle_{\cS^2(\cH)},\qquad T\in\cS^2(\cH), \]
where \(R\in\cS^2(\cH)\smo\) is fixed. Note however that questions beyond Hilbert--Schmidt regularity are much more delicate, especially in non-unimodular settings, as they are connected to notions like admissibility \cite{BBLS22} or Duflo--Moore conditions. 

The difficult problem, as one might expect, is the one concerning the diagonal vector coefficient of $\rho$, that is the autovoice (``ambiguity'') transform
\[ A_\rho f(\xi)\coloneqq \langle f,\rho(\xi)f\rangle_{\cH}, \qquad f \in \cH, \,\ \xi \in \Xi. \] 
Note that, by Cauchy--Schwarz, we have $|A_\rho f(\xi)| \le \|f\|_{\cH}^2$, which in turn yields the universal concentration bound
\[ \|A_\rho f\|_{L^p(\Omega)}\le \mu_{\Xi}(\Omega)^{1/p}, \qquad \|f\|_{\cH}=1. \]
The Heisenberg ambiguity function is a positive special case of this diagonal problem, in light of \cite{Nicola-Romero-Trapasso-2022}, but inspecting the delicate proof suggests that analogous conclusions should not be expected in general. Indeed, this concentration problem seems governed not by compactness of a fixed window: If \(\rho\) contains a one-dimensional invariant subrepresentation, then the universal bound is attained on the corresponding unit vector, and conversely (in connected groups) any exact optimizer for the sharp bound forces such a one-dimensional subrepresentation. The following result makes this mechanism precise.

\begin{prop} \label{prop:DiagonalWeakContainment}
Let \(\Xi\) be a second countable locally compact group with left Haar measure \(\mu_\Xi\), and let \(\rho \colon \Xi\to\mathcal U(\cH)\) be a strongly continuous unitary representation. Assume that there exists a continuous homomorphism \(\chi \colon \Xi\to\mathbb T\) such that, for every compact set \(K\subset\Xi\) and every \(\varepsilon>0\), there is a unit vector \(f\in\cH\) satisfying
\[\sup_{\xi\in K}|A_\rho f(\xi)-\chi(\xi)|<\varepsilon.\]
Then, for every \(1\le p<\infty\) and every measurable set \(\Omega\subset\Xi\) with \(0<\mu_\Xi(\Omega)<\infty\), one has
\[\sup_{\|f\|_{\cH}=1}\|A_\rho f\|_{L^p(\Omega)}=\mu_\Xi(\Omega)^{1/p}.\]
In addition, if \(\Xi\) is connected and \(\rho\) contains no one-dimensional invariant subrepresentation, then the supremum is not attained.
\end{prop}

\begin{proof}
The upper bound is clear, so let us prove the lower bound. Let \(\varepsilon>0\) and, by inner regularity of the Haar measure, choose a compact set \(K\subset\Omega\) such that \(\mu_\Xi(K)>\mu_\Xi(\Omega)-\varepsilon\). By assumption, there exists \(f\in\cH\) with \(\|f\|_{\cH}=1\) such that
\[\sup_{\xi\in K}|A_\rho f(\xi)-\chi(\xi)|<\varepsilon.\]
Since \(|\chi(\xi)|=1\), it follows that \(|A_\rho f(\xi)|\ge 1-\varepsilon\) for every \(\xi\in K\), and thus 
\[\|A_\rho f\|_{L^p(\Omega)}\ge \|A_\rho f\|_{L^p(K)}\ge (1-\varepsilon)\mu_\Xi(K)^{1/p}.\]
Letting \(\varepsilon\downarrow0\) yields the desired formula for the optimal value.

It remains to prove nonattainment under connectedness. Suppose by contradiction that a unit vector \(f\in\cH\) attains the supremum. Since \(|A_\rho f|\le1\), equality in the \(L^p\)-bound implies $|A_\rho f(\xi)|=1$ for almost every \(\xi\in\Omega\). Set then
\[E=\{\xi\in\Omega:|A_\rho f(\xi)|=1\},\]
so that \(\mu_\Xi(E)>0\). Equality holds in Cauchy--Schwarz for every \(\xi\in E\), hence \(\rho(\xi)f\in\C f\). We are then led to consider the projective stabilizer
\[H_f=\{\xi\in\Xi : \rho(\xi)f\in\C f\}.\]
Note that \(H_f\) is a closed subgroup of \(\Xi\) and \(E\subset H_f\). By Steinhaus' theorem, \(EE^{-1}\) contains a neighbourhood of the identity, and since \(EE^{-1}\subset H_f\) we have that the subgroup \(H_f\) is open. If \(\Xi\) is connected then every open subgroup coincides with \(\Xi\), so \(H_f=\Xi\). As a consequence, \(\C f\) is a one-dimensional invariant subrepresentation of \(\rho\), contradicting the assumption.
\end{proof}

\begin{remark}
Note that the assumptions in Proposition~\ref{prop:DiagonalWeakContainment} coincide with the coefficient form of weak containment of a one-dimensional subrepresentation. There is a harmless conjugation to keep in mind: Recalling that the Hilbert space inner product is meant to be linear in the first entry, an actual subrepresentation \(\rho(\xi)u=\eta(\xi)u\) gives \(A_\rho u(\xi)=\overline{\eta(\xi)}\). 
\end{remark}

The previous result can be now specialized to several representations of interest in time-frequency analysis, such as those behind the wavelet and the shearlet transform. We examine the wavelet case in more detail, since the shearlet one follows by similar arguments after suitable adaptations to the different geometry. Consider then the affine group
\[ \Xi_{\mathrm{aff}}=\R\rtimes\R_+,\qquad (b,a)(b',a')=(b+ab',aa'), \]
with left Haar measure \(\dd{\mu_{\Xi_{\mathrm{aff}}}}(b,a)=a^{-2} \dd{b} \dd{a}\). The standard affine wavelet representation on \(L^2(\R)\) is
\[\rho(b,a)f(x)=a^{-1/2}f\left(\frac{x-b}{a}\right),\qquad (b,a)\in\Xi_{\mathrm{aff}}.\]
For fixed \(g\in L^2(\R)\), the corresponding affine voice transform is then
\[V_gf(b,a)=\langle f,\rho(b,a)g\rangle_{L^2(\R)}=a^{-1/2}\int_{\R}f(x)\overline{g\left(\frac{x-b}{a}\right)} \dd{x},\]
which coincides with the standard continuous wavelet transform of $f$ when $g$ is admissible. The corresponding scalogram \(|V_gf|^2\) is the rank-one compact-window affine QHA distribution associated with \(g\otimes g \in \cS^1\). 

It is useful to recall for later use the Fourier-side form of the same representation:
\[\widehat{\rho(b,a)f}(\omega)=a^{1/2}e^{-2\pi i b\omega}\widehat f(a\omega).\]
The sign of \(\omega\) is thus preserved, and \(L^2(\R)\) decomposes into the invariant subspaces
\[\mathcal H_+=\{f\in L^2(\R):\operatorname{ess\,supp}\widehat f\subset\R_+\},\qquad \mathcal H_-=\{f\in L^2(\R):\operatorname{ess\,supp}\widehat f\subset\R_-\}.\]
The restrictions of \(\rho\) to \(\mathcal H_+\) and \(\mathcal H_-\) are the two usual irreducible wavelet components. Note that the positive-frequency component is unitarily equivalent to the normalization used in~\cite{BBLS22}, namely \(L^2(\R_+,\dd\omega/\omega)\) with \(U(x,a)H(\omega)=e^{2\pi i x\omega}H(a\omega)\). Indeed, if \(\mathcal F_+f=\widehat f|_{\R_+}\) and \(Ch(\omega)=\omega^{1/2}h(\omega)\), then \(C\mathcal F_+:\mathcal H_+\to L^2(\R_+,\dd\omega/\omega)\) is unitary and satisfies \((C\mathcal F_+)\rho_+(b,a)(C\mathcal F_+)^{-1}=U(-b,a)\), where \(\rho_+=\rho|_{\mathcal H_+}\). We shall formulate the result for the standard representation \(\rho\), but use the Fourier-side realization in the proof, as in~\cite{BBLS22}.

We stress that, although the full representation \(\rho\) on \(L^2(\R)\) is reducible, it has no one-dimensional invariant subrepresentation. Indeed, if \(\C f\) were invariant, then at least one of the projections of \(f\) onto the reducing subspaces \(\mathcal H_+\) or \(\mathcal H_-\) would be nonzero and, since the corresponding projection commutes with \(\rho\), would span a one-dimensional invariant subspace of the corresponding irreducible component. This is impossible, because both components are infinite-dimensional and irreducible. The absence of one-dimensional subrepresentations is the ultimate representation-theoretic nonattainment mechanism behind the next result, concerning optimal concentration for the wavelet ambiguity transform:
\[ A_\rho f(b,a)=\langle f,\rho(b,a)f\rangle_{L^2(\R)}. \]

\begin{prop}
For every \(1\le p<\infty\) and every measurable \(\Omega\subset\Xi_{\mathrm{aff}}\) with \(0<\mu_{\Xi_{\mathrm{aff}}}(\Omega)<\infty\), one has
\[ \sup_{\|f\|_{L^2(\R)}=1}\|A_\rho f\|_{L^p(\Omega)}=\mu_{\Xi_{\mathrm{aff}}}(\Omega)^{1/p}, \]
but the supremum is not attained.
\end{prop}

\begin{proof}
By Proposition~\ref{prop:DiagonalWeakContainment}, it is enough to verify the diagonal-coefficient approximation for the trivial character \(\Xi_{\mathrm{aff}}\to \mathbb T\), and that \(\rho\) has no one-dimensional invariant subrepresentation. The second issue has already been discussed, so we are left with the first one. To this aim, we will construct diagonal coefficients converging to \(1\) uniformly on compact subsets. 
First, let us define \(f_n\in L^2(\R)\) by prescribing its Fourier transform \(h_n=\widehat f_n\), with \(h_n\) supported in \(\R_+\). Choose a real sequence \(M_n\to +\infty\) and set \(I_n=[e^{-M_n-n},e^{-M_n}]\), then consider\footnote{Intuitively, the sequence \(f_n\) is designed so that its \(L^2\)-mass is uniformly spread over a logarithmic frequency interval of length \(n\) drifting to \(0\), with the purpose of making bounded dilations and translations almost invisible.}
\[ h_n(\omega)=\frac{1}{\sqrt n}\omega^{-1/2}\mathbf 1_{I_n}(\omega),\qquad \omega>0, \]
and set \(\widehat f_n(\omega)=0\) for \(\omega\le0\). By Plancherel's theorem we have \(\|f_n\|_{L^2(\R)}=\|h_n\|_{L^2(\R_+)}=1\). Moreover, since \(h_n\) is real and nonnegative, we compute
\[A_\rho f_n(b,a) = a^{1/2}\int_0^\infty e^{2\pi i b\omega}h_n(\omega)\overline{h_n(a\omega)} \dd{\omega}.\]
The integrand is nonzero exactly when \(\omega\in I_n\) and \(a\omega\in I_n\), that is, \(\omega\in I_n\cap a^{-1}I_n\). On this set it holds
\[ h_n(\omega)h_n(a\omega)=\frac1n\omega^{-1/2}(a\omega)^{-1/2}=\frac1n a^{-1/2}\omega^{-1}, \]
therefore leading to 
\[ A_\rho f_n(b,a)=\frac1n\int_{I_n\cap a^{-1}I_n}e^{2\pi i b\omega}\frac{\dd{\omega}}{\omega}. \]
Let \(K\subset \Xi_{\mathrm{aff}}\) be compact, so that there are constants \(B,L>0\) such that \(|b|\le B\) and \(|\log a|\le L\) for every \((b,a)\in K\). For \(n>L\), the change of variables \(s=\log\omega\) maps $I_n$ onto $J_n=[-M_n-n,-M_n]$ and produces 
\[ \int_{I_n\cap a^{-1}I_n}\frac{\dd{\omega}}{\omega}=|J_n\cap(J_n-\log a)|=n-|\log a|. \]
As a result, for \((b,a)\in K\) we have 
\[
\begin{aligned}
\left|A_\rho f_n(b,a)-1\right|  &=  \left| \frac1n\int_{I_n\cap a^{-1}I_n}e^{2\pi i b\omega}\frac{\dd{\omega}}{\omega}-1 \right| \\
&\le \frac1n\int_{I_n\cap a^{-1}I_n} \left|e^{2\pi i b\omega}-1\right|\frac{\dd{\omega}}{\omega} + \left|
\frac1n\int_{I_n\cap a^{-1}I_n}\frac{\dd{\omega}}{\omega}-1 \right| .
\end{aligned}
\]
The second term is bounded by \(L/n\), while the first term can be controlled after noting that \(\omega\in I_n\) implies \(\omega\le e^{-M_n}\), hence 
\[ |e^{2\pi i b\omega}-1|\le 2\pi |b|\omega\le 2\pi B e^{-M_n}, \]
and thus 
\[ \sup_{(b,a)\in K}|A_\rho f_n(b,a)-1|\le 2\pi B e^{-M_n}+\frac{L}{n}. \]
Since \(M_n\to+\infty\), we have proved that \(A_\rho f_n\to1 \) uniformly on compact subsets of \(\Xi_{\mathrm{aff}}\), and thus the claim by virtue of Proposition~\ref{prop:DiagonalWeakContainment}. 
\end{proof}

\begin{remark}
As anticipated, the nonattainment mechanism is governed by irreducibility --- through the irreducible frequency components rather than through irreducibility of the full standard representation. In the claim we kept the formulation on \(L^2(\R)\), since this is the usual domain of the continuous wavelet transform, but if one wants a purely irreducible version it suffices to restrict the whole discussion to \(\mathcal H_+\) and replace \(\rho\) by \(\rho_+=\rho|_{\mathcal H_+}\), as is standard in affine QHA~\cite{BBLS22}. In that case, for every \(1\le p<\infty\),
\[ \sup_{\|f\|_{\mathcal H_+}=1}\|A_{\rho_+}f\|_{L^p(\Omega)} = \mu_{\Xi_{\mathrm{aff}}}(\Omega)^{1/p}, \]
and the supremum is not attained. The proof is exactly the positive-frequency construction given above --- in fact, the maximizing sequence already lies in \(\mathcal H_+\).
\end{remark}

\begin{remark}
The endpoint case \(p=\infty\) is slightly different. The optimal value is 
\[ \sup_{\|h\|_{L^2}=1}\|A_\rho h\|_{L^\infty(\Omega)}=1, \]
but attainment depends on the position of \(\Omega\) --- more precisely, the supremum is attained if and only if $e=(0,1)\in\operatorname{supp}(\mathbf 1_\Omega \dd{\mu_{\Xi_{\mathrm{aff}}}})$. 

Indeed, if the identity \(e=(0,1)\) belongs to the measure-theoretic support of \(\Omega\), then \(A_\rho h(e)=1\) for every unit vector \(h\), hence giving an \(L^\infty\)-optimizer by continuity. 

Conversely, suppose that \(e\notin\operatorname{supp}(\mathbf 1_\Omega \dd{\mu_{\Xi_{\mathrm{aff}}}})\), so that there exists a neighborhood \(U\) of \(e\) such that $\mu_{\Xi_{\mathrm{aff}}}(\Omega\cap U)=0$. Assume by contradiction that a unit vector \(h\) attains the \(L^\infty\)-value, namely $\|A_\rho h\|_{L^\infty(\Omega)}=1$. For every \(k\in\N\), set
\[ E_k=\{(b,a)\in\Omega:\ |A_\rho h(b,a)|>1-1/k\}. \]
Since \(\|A_\rho h\|_{L^\infty(\Omega)}=1\), each \(E_k\) has positive Haar measure. The same holds for \(E_k\setminus U\), since \(\mu_{\Xi_{\mathrm{aff}}}(\Omega\cap U)=0\). Choose then \((b_k,a_k)\in E_k\setminus U\), so
\[ |A_\rho h(b_k,a_k)|>1-\frac1k. \]
A standard argument\footnote{It follows from \[ A_\rho h(b,a)=a^{1/2}\int_{\R}e^{2\pi i b\omega}\widehat h(\omega)\overline{\widehat h(a\omega)} \dd{\omega}, \]
the Riemann--Lebesgue lemma in $b$ and approximation of \(\widehat h\) by compactly supported functions away from \(0\) and \(+\infty\).} gives \(A_\rho h\in C_0(\Xi_{\mathrm{aff}})\), so the level set
\[ K_h=\{(b,a)\in\Xi_{\mathrm{aff}}: |A_\rho h(b,a)|\ge 1/2\} \]
is compact. For \(k\) large enough one has \((b_k,a_k)\in K_h\setminus U\), and up to subsequences we may assume $(b_k,a_k)\to(b_0,a_0)\in K_h\setminus U$. By continuity and equality in Cauchy--Schwarz's inequality we have $|A_\rho h(b_0,a_0)|=1$ and $\rho(b_0,a_0)h\in\C h$. It is well known that the standard affine wavelet representation has no nonzero eigenvectors for nonidentity group elements\footnote{If \(\rho(b,a)h=\lambda h\) with \(h\ne0\), then \((b,a)=e\). Indeed, take Fourier transforms and write \(H=\widehat h\). Then $a^{1/2}e^{-2\pi i b\omega}H(a\omega)=\lambda H(\omega)$. If \(a=1\) and \(b\ne0\), this identity forces \(H\) to be supported on a level set of \(e^{-2\pi i b\omega}\), which has measure zero, a contradiction. If \(a\ne1\), taking absolute values gives $|H(a\omega)|^2=a^{-1}|H(\omega)|^2$, implying \(h=0\): If $a>1$ the shells \(a^k[1,a)\), \(k\in\Z\), are disjoint and cover \(\R_+\), but are forced to carry the same \(L^2\)-mass, which must be zero; the case \(0<a<1\) follows by replacing \(a\) with \(a^{-1}\), and similar arguments on \(-a^k[1,a)\), \(k\in\Z\), rule out mass on \(\R_-\).}, therefore the contradictory conclusion \(K_h \setminus U \ni (b_0,a_0)=e\). 
\end{remark}

\begin{remark}
As anticipated, the same mechanism applies to the standard connected continuous shearlet representation \cite{shearlet1}. Let \(A_a=\mathrm{diag}(a,\sqrt a)\), \(S_s=\begin{pmatrix}1&s\\0&1\end{pmatrix}\), and \(\Xi_{\mathrm{sh}}=\R^2\rtimes\{S_sA_a:s\in\R,\ a>0\}\). For the standard representation
\[\rho(x,s,a)f(y)=|\det(S_sA_a)|^{-1/2}f((S_sA_a)^{-1}(y-x)),\]
the only change with respect to the affine wavelet proof is the construction of the weakly contained vectors. On the Fourier side one can use the coordinates \(t=\log\xi_1\), \(r=\xi_2/\xi_1\) on the half-plane \(\xi_1>0\), where the dual shearlet action is 
\[(t,r)\mapsto(t+\log a,a^{-1/2}(r+s)).\] 
The logarithmic intervals \(I_n\) in the wavelet proof are thus replaced by the parabolic boxes
\[E_n=\{(\xi_1,\xi_2):\xi_1=e^t,\ \xi_2=re^t,\ t\in[-M_n-n,-M_n],\ |r|\le e^{-t/2}\}.\]
Setting \(\dd{\nu(\xi)}=\xi_1^{-3/2}\dd{\xi}\), these boxes have measure \(\nu(E_n)=2n\) and the functions \(\widehat f_n=(2n)^{-1/2}\xi_1^{-3/4}\mathbf 1_{E_n}\) satisfy \(A_\rho f_n\to1\) uniformly on compact subsets of \(\Xi_{\mathrm{sh}}\), by the same overlap argument used in the wavelet case. Proposition~\ref{prop:DiagonalWeakContainment} therefore applies, the supremum not being attained since the translation subgroup excludes one-dimensional invariant subrepresentations.
\end{remark}

The same weak-containment principle lifts to operator level. Define 
\[\alpha_\xi(T)=\rho(\xi)T\rho(\xi)^*,\qquad D_\rho T(\xi)=\langle T,\alpha_\xi(T)\rangle_{\cS^2(\cH)}.\] 
This coincides with the diagonal coefficient of the conjugation representation \(\alpha\) on \(\cS^2(\cH)\). Applying Proposition~\ref{prop:DiagonalWeakContainment} to \(\alpha\) gives the following immediate consequence.

\begin{corollary}\label{cor:OperatorDiagonalWeakContainment}
Assume that there exists a continuous homomorphism \(\chi \colon \Xi\to\mathbb T\) such that, for every compact set \(K\subset\Xi\) and every \(\varepsilon>0\), there is \(T\in\mathcal S^2(\mathcal H)\) with \(\|T\|_{\mathcal S^2}=1\) satisfying 
\[ \sup_{\xi\in K}|D_\rho T(\xi)-\chi(\xi)|<\varepsilon . \] 
Then, for every \(1\le p<\infty\) and every measurable \(\Omega\subset\Xi\) with \(0<\mu_\Xi(\Omega)<\infty\), \[\sup_{\|T\|_{\cS^2}=1}\|D_\rho T\|_{L^p(\Omega)}=\mu_\Xi(\Omega)^{1/p}.\] If, in addition, \(\Xi\) is connected and \(\alpha\) contains no one-dimensional invariant subrepresentation on \(\cS^2(\cH)\), then the supremum is not attained.
\end{corollary}

\begin{remark}
It is important to distinguish the operator-level problem from the related scalar one. Consider for instance the Heisenberg case: Total correlation is optimized over all Hilbert--Schmidt operators, whereas the ambiguity problem corresponds only to the rank-one restriction. Indeed, for \(T=f\otimes f\), one has
\[ D_\pi(f\otimes f)(z)=|A_\pi f(z)|^2, \qquad  \|D_\pi(f\otimes f)\|_{L^p(\Omega)}=\|A_\pi f\|_{L^{2p}(\Omega)}^2.\]
The sharp Cauchy--Schwarz value for total correlation is thus obtained only after allowing general Hilbert--Schmidt operators, while the rank-one subclass remains subject to the uncertainty constraints responsible for the strict suboptimality of the ambiguity problem, for instance on discs in the superquadratic regime. This is also consistent with Proposition~\ref{prop:DiagonalWeakContainment}, since the Schr\"odinger representation of the Heisenberg group has a nontrivial central phase, and therefore fails to satisfy the weak containment assumption of that result, while in the conjugation representation \(T\mapsto\pi(z)T\pi(z)^*\) the projective phase cancels and thus the corresponding representation \(\pi\otimes\overline{\pi}\) may weakly contain the trivial representation. 

These remarks point towards a broader picture, since weak containment of the trivial representation for the conjugation representation, identified with \(\rho\otimes\overline{\rho}\) after the standard Hilbert--Schmidt tensor identification, is a condition closely connected to amenability for unitary representations in the sense of Bekka \cite{bekka}, and spectral-gap phenomena. We plan to explore these and related aspects in future investigations. 
\end{remark}

We can specialize to the affine QHA setting: For \(T\in\cS^2(L^2(\R_+))\), we have
\[ D_\rho T(b,a)\coloneqq \langle T,\alpha_{(b,a)}(T)\rangle_{\cS^2},\qquad \alpha_{(b,a)}(T)=\rho(b,a)T\rho(b,a)^*. \]
Here \(\rho\) denotes the Fourier-side positive-frequency affine representation, unitarily equivalent to the logarithmic-frequency normalization used in~\cite{BBLS22}. This is the diagonal operator-correlation associated with the affine-QHA conjugation action. Equivalently, it is the autocorrelation case of the affine coefficient \(C_R T(b,a)=\langle T,\alpha_{(b,a)}(R)\rangle_{\cS^2}\). 

Corollary~\ref{cor:OperatorDiagonalWeakContainment} applies directly, since the weak-containment hypothesis is verified by the rank-one operators \(T_n=f_n\otimes f_n\) with \(f_n\) constructed above:
\[D_\rho T_n(b,a)=|A_\rho f_n(b,a)|^2\to1,\quad \text{uniformly on compact subsets of } \Xi_{\mathrm{aff}}.\]
The absence of one-dimensional invariant subrepresentations follows as usual from irreducibility: If \(\alpha_{(b,a)}(T)=\chi(b,a)T\) for some nonzero \(T\in\cS^2(L^2(\R_+))\) and some continuous homomorphism \(\chi:\Xi_{\mathrm{aff}}\to\mathbb T\), then \(T^*T\) commutes with \(\rho(b,a)\) for every \((b,a)\), hence \(T^*T=cI\) by Schur's lemma. Since \(L^2(\R_+)\) is infinite-dimensional, this is impossible for a nonzero Hilbert--Schmidt operator.

To summarize, while the compact-window QHA mechanism is not specific to the Euclidean Heisenberg representation and fixed-window concentration obeys a general compactness phenomenon, diagonal voice concentration splits into distinct regimes. The affine example shows that the positive Heisenberg ambiguity theorem is not a high-level consequence of coorbit or QHA covariance: The representation-theoretic machinery gives fixed-window compactness and the strict-gap formalism, but the size of the defect level is ultimately representation-dependent. To be precise, if the equality $\Lambda_{p,\Omega}(\rho)=\Lambda^{\mathrm{ess}}_{p,\Omega}(\rho)$ holds, attainment may still occur only when the maximizing defects arise from exact symmetries that can be recentered, as in the Heisenberg ambiguity problem, but also with a highly nontrivial analytic input: Local coorbit bounds are needed to convert nonvanishing concentration into a nonzero weak profile after time-frequency recentering \cite{Nicola-Romero-Trapasso-2022}, or careful control of representation-dependent phenomena as in \cite{Stra-Svela-Trapasso-2025}. The affine wavelet representation falls outside this class: The logarithmically spread sequence $f_n$ is weakly null and satisfies \(A_\rho f_n\to1\) locally uniformly, hence the defect level is maximal and no optimizer exists.

We conclude by highlighting that infinite-dimensional irreducibility by itself does not decide the attainment problem: The affine wavelet representation gives nonattainment, whereas there are also irreducible infinite-dimensional representations for which the diagonal problem attains optimal concentration for suitable observation sets. Consider for instance \(\Xi=N\rtimes\Z\), where \(N\) is a countable discrete abelian group, and take a character \(\chi\in\widehat N\) whose orbit under the dual \(\Z\)-action has trivial stabilizer. The associated induced representation on \(\ell^2(\Z)\) is irreducible and can be written in diagonal form as $\rho(n,0)e_j=\chi_j(n)e_j$, where \(\chi_j\) is the \(j\)-th orbit translate of \(\chi\). Now, if \(\Omega\subset N\times\{0\}\) is finite and nonempty, then for every \(j\in\Z\) we have
\[ |A_\rho e_j(n,0)|=|\langle e_j,\rho(n,0)e_j\rangle|=1,\qquad (n,0)\in\Omega. \]
As a consequence, for every \(1\le p<\infty\) the optimal value is achieved: $\|A_\rho e_j\|_{L^p(\Omega)}=|\Omega|^{1/p}$. 

\section*{Acknowledgments}
We are happy to thank Franz Luef for discussions and encouragement. We also gratefully acknowledge inspiring discussions over the years on the topics of this manuscript with the coauthors of the series: Fabio Nicola, José Luis Romero, and Federico Stra.

Original motivations for this work date back to the first QHA workshop held in Trondheim (June 2023). The first author express his gratitude to Politecnico di Torino for hospitality on the occasion of a long visit in Spring 2025, during which the work started. 
\printbibliography
\end{document}